\documentclass[a4paper,12pt]{article}
\usepackage{latexsym}
\usepackage{amsmath}
\usepackage{amssymb}
\usepackage{enumerate}
\usepackage{bm}
\usepackage{mathrsfs}

\usepackage{theorem}
\newtheorem{theorem}{Theorem}[section]

\newtheorem{lemma}[theorem]{Lemma}
\newtheorem{corollary}[theorem]{Corollary}
\newtheorem{case}{Case}

\theorembodyfont{\rmfamily}
\newtheorem{proof}{\textmd{\textit{Proof.}}}

\newtheorem{remark}[theorem]{Remark}
\newtheorem{example}[theorem]{Example}
\newtheorem{definition}[theorem]{Definition}
\newtheorem{assumption}[theorem]{Assumption}

\makeatletter

\@addtoreset{equation}{section}
\makeatother

\newcommand{\qedd}{\hfill \Box}
\newcommand{\ve}{\varepsilon}
\newcommand{\del}{\partial}
\newcommand{\lra}{\longrightarrow}

\newcommand{\N}{\ensuremath{\mathbb{N}}}

\newcommand{\R}{\ensuremath{\mathbb{R}}}
\newcommand{\Sph}{\ensuremath{\mathbb{S}}}

\def\diam{\mathop{\mathrm{diam}}\nolimits}

\def\CAT{\mathop{\mathrm{CAT}}\nolimits}

\title{Gradient flows and a Trotter--Kato formula\\ of semi-convex functions on CAT(1)-spaces}
\author{Shin-ichi Ohta\thanks{Department of Mathematics, Kyoto University,
Kyoto 606-8502, Japan ({\sf sohta@math.kyoto-u.ac.jp});
Supported in part by the Grant-in-Aid for Young Scientists (B) 23740048.}
\ \&\ Mikl\'os P\'alfia\thanks{Department of Mathematics, Kyoto University,
Kyoto 606-8502, Japan ({\sf palfia.miklos@aut.bme.hu});
Supported in part by the SGU project of Kyoto University, the JSPS international research fellowship,
JSPS KAKENHI Grant No.~14F04320,
the National Research Foundation of Korea (NRF) Grant No.~2015R1A3A2031159 funded by the Korea government (MEST),
and the Hungarian National Research Fund OTKA K-115383.}}
\date{}
\begin{document}

\maketitle

\begin{abstract}
We generalize the theory of gradient flows of semi-convex functions on $\CAT(0)$-spaces,
developed by Mayer and Ambrosio--Gigli--Savar\'e, to $\CAT(1)$-spaces.
The key tool is the so-called ``commutativity'' representing a Riemannian nature of the space,
and all results hold true also for metric spaces satisfying the commutativity
with semi-convex squared distance functions.
Our approach combining the semi-convexity of the squared distance function
with a Riemannian property of the space seems to be of independent interest,
and can be compared with Savar\'e's work on the local angle condition under lower curvature bounds.
Applications include the convergence of the discrete variational scheme to a unique gradient curve,
the contraction property and the evolution variational inequality of the gradient flow,
and a Trotter--Kato product formula for pairs of semi-convex functions.
\end{abstract}

\section{Introduction}

The theory of gradient flows in singular spaces is a field of active research having applications in various fields.
For instance, regarding heat flow as gradient flow of the relative entropy in the ($L^2$-)Wasserstein space,
initiated by Jordan, Kinderlehrer and Otto~\cite{JKO},
is known as a useful technique in partial differential equations
(see \cite{Ot}, \cite{Vi1}, \cite{AGSbook}, \cite{ASZ} among many others),
and has played a crucial role in the recent remarkable development of geometric analysis on
metric measure spaces satisfying the (\emph{Riemannian}) \emph{curvature-dimension condition}
(see \cite{Vi}, \cite{Gi}, \cite{GKO}, \cite{AGShf}, \cite{AGSrcd}, \cite{EKS}).
One of the recent striking achievements is Gigli's splitting theorem \cite{Gsplit}
(see also a survey \cite{Gsplit'}), in which the gradient flow of the Busemann function is used in an impressive way.

We follow the strategy of constructing a gradient flow of a lower semi-continuous, semi-convex function
$\phi$ on a metric space $(X,d)$ via the discrete variational scheme employing
the \emph{Moreau--Yosida approximation}:
\begin{equation}\label{eq:MY}
\phi_{\tau}(x):=\inf_{z \in X} \left\{ \phi(z) +\frac{d^2(x,z)}{2\tau} \right\},
 \quad x \in X,\ \tau>0.
\end{equation}
A point $x_{\tau}$ attaining the above infimum is considered as an approximation
of the point $\xi(\tau)$ on the gradient curve $\xi$ of $\phi$ with $\xi(0)=x$.
Here $\phi$ is said to be \emph{semi-convex} if it is \emph{$\lambda$-convex} for some $\lambda \in \R$
meaning that
\[ \phi\big( \gamma(t) \big) \le (1-t)\phi\big( \gamma(0) \big) +t\phi\big( \gamma(1) \big)
 -\frac{\lambda}{2}(1-t)td^2\big(\gamma(0),\gamma(1) \big) \]
along geodesics $\gamma:[0,1] \lra X$.
Since the approximation scheme is based on the distance function,
finer properties of the distance function provide finer analysis of gradient flows.
In \cite{Jo}, \cite{Ma} (with applications to harmonic maps)
and \cite{AGSbook} (with applications to the Wasserstein spaces),
gradient flows in \emph{$\CAT(0)$-spaces} (non-positively curved metric spaces) are well investigated;
see also Ba\v{c}\'ak's recent book \cite{Ba-book}.
There the $2$-convexity of the squared distance function,
that is indeed the definition of $\CAT(0)$-spaces, played essential roles.
We shall generalize their theory to \emph{$\CAT(1)$-spaces} (metric spaces of sectional curvature $\le 1$),
where the distance functions are only semi-convex.
$\CAT(1)$-spaces can have a more complicated global structure than $\CAT(0)$-spaces.
For instance, all $\CAT(0)$-spaces are contractible while $\CAT(1)$-spaces may not.

It is known that the direct application of the techniques of $\CAT(0)$-spaces to $\CAT(1)$-spaces
does not work.
The point is that the $K$-convexity of the squared distance function with $K<2$ 
holds true on some non-Hilbert Banach spaces,
on those the behavior of gradient flows is much less understood.
We overcome this difficulty by introducing the notion of ``commutativity''
representing a ``Riemannian nature'' of the space.
Precisely, the key ingredients of our analysis are the following properties of $\CAT(1)$-spaces:
\begin{enumerate}[(A)]
\item
The \emph{commutativity}:
\begin{equation}\label{eq:commu}
\lim_{s \downarrow 0} \frac{d^2(\gamma(s),z)-d^2(x,z)}{s}
 = \lim_{t \downarrow 0} \frac{d^2(\eta(t),y)-d^2(x,y)}{t}
\end{equation}
for geodesics $\gamma$ and $\eta$ with $x=\gamma(0)=\eta(0)$, $\gamma(1)=y$ and $\eta(1)=z$
(see \eqref{eq:comm} in the proof of Lemma~\ref{lm:key});

\item
The semi-convexity of the squared distance function (see Lemma~\ref{lm:Kconv}).
\end{enumerate}
(One can more generally consider some other family of curves along those (A) and (B) hold,
as was essentially used to study the Wasserstein spaces in \cite{AGSbook}.)
This approach, reinforcing the semi-convexity with the Riemannian nature of the space,
seems to be of independent interest and is in a similar spirit to Savar\'e's work~\cite{Sa}
based on the semi-concavity of distance functions and the \emph{local angle condition},
those properties fit the study of spaces with lower curvature bounds.
The semi-convexity and semi-concavity are usually studied in the context of
Banach space theory and Finsler geometry, see \cite{BCL}, \cite{Oconv}, \cite{Ouni}.
In the $\CAT(0)$-setting, the commutativity follows from the $2$-convexity of the squared distance function
and the role of the commutativity is implicit.
The commutativity is not true in non-Riemannian Finsler manifolds (see Remark~\ref{rm:key}(b)).
Actually, the lack of the commutativity is the reason why the first author and Sturm introduced
the notion of \emph{skew-convexity} in \cite{OSnc} to study the contraction property
of gradient flows in Finsler manifolds.
In contrast, on the Wasserstein space over a Riemannian manifold,
we have a sort of Riemannian structure but the convexity of the squared distance function fails.
See \cite[\S B]{Gsplit} for a connection between \eqref{eq:commu} and
the \emph{infinitesimal Hilbertianity} which is an ``almost everywhere'' notion of Riemannian nature.

Modifying the calculation in \cite{AGSbook} with the help of the commutativity,
we arrive at the key estimate for a $\lambda$-convex function $\phi$
and a $K$-convex distance (see Lemma~\ref{lm:key}):
\[ d^2(x_{\tau},y)
 \le d^2(x,y) -\lambda \tau d^2(x_{\tau},y) +2\tau\{ \phi(y)-\phi(x_{\tau}) \}
 -\frac{K}{2}d^2(x,x_{\tau}), \]
where $x_{\tau}$ is a point attaining the infimum in \eqref{eq:MY}.
Surprisingly, even with $K<0$ (as well as $\lambda<0$), this estimate is enough to generalize the argument of \cite{AGSbook}.
We show the convergence of the discrete variational scheme to a unique gradient curve (Theorem~\ref{th:T422}),
the contraction property (Theorem~\ref{th:P431}) and the evolution variational inequality (Theorem~\ref{th:T432})
of the gradient flow.
Moreover, along the lines of \cite{Ma} and \cite{CM},
we study the large time behavior of the flow (\S \ref{ssc:Mayer})
and prove a Trotter--Kato product formula for pairs of semi-convex functions (Theorem~\ref{th:TK}).
The latter is a two-fold generalization of the existing results in \cite{CM}, \cite{St} and \cite{Ba}
for convex functions on $\CAT(0)$-spaces
(to be precise, an inequality corresponding to our key estimate with $\lambda=0$ and $K=2$
is an assumption of \cite{CM}).
We stress that we use only the qualitative properties of $\CAT(1)$-spaces instead of the direct curvature condition.
Thus our technique also applies to every metric space satisfying the conditions (A) and (B) above,
under a mild coercivity assumption on $\phi$
(see Case~\ref{case:Kcon} at the beginning of Section~\ref{sc:gf}).
This case could be more important than the $\CAT(1)$-setting,
because in $\CAT(1)$-spaces the squared distance function is locally $K$-convex with $K>0$,
that makes some discussions easier with the help of the globalization technique (see \S \ref{ssc:pi}).

We finally mention some more related works.
Gradient flow in metric spaces with lower sectional curvature bounds (\emph{Alexandrov spaces})
is investigated in \cite{PP}, \cite{Ly}.
This technique was generalized to the Wasserstein spaces over Alexandrov spaces in \cite{Ogra}
and \cite{Sa}.
Sturm~\cite{Stu} recently studied gradient flows in metric measure spaces satisfying
the Riemannian curvature-dimension condition.
Discrete-time gradient flow is also an important subject related to optimization theory,
for that we refer to \cite{OP} and the references therein.

The organization of the article is as follows.
In Section~\ref{sc:prel}, we recall preliminary results on gradient flows in metric spaces from \cite{AGSbook},
followed by the necessary facts of $\CAT(1)$-spaces.
Section~\ref{sc:key} is devoted to our key estimate.
We apply the key estimate to the study of gradient flows in Section~\ref{sc:gf},
and prove a Trotter--Kato product formula in Section~\ref{sc:TK}.

\section{Preliminaries}\label{sc:prel}

Let $(X,d)$ be a complete metric space.
A curve $\gamma:[0,1] \lra X$ is called a \emph{geodesic} if it is locally minimizing and of constant speed.
We call $\gamma$ a \emph{minimal} geodesic if it is globally minimizing,
namely $d(\gamma(s),\gamma(t))=|s-t|d(\gamma(0),\gamma(1))$ for all $s,t \in [0,1]$.
We say that $(X,d)$ is \emph{geodesic}
if any two points $x,y \in X$ admit a minimal geodesic between them.

\subsection{Gradient flows in metric spaces}\label{ssc:gf}

We recall basic facts on the construction of gradient curves in metric spaces.
We follow the technique called the ``minimizing movements'' going back to (at least) De Giorgi~\cite{DG},
see \cite{AGSbook} for more on this theory.
We also refer to \cite{Br}, \cite{CL} for classical theories on linear spaces.

\subsubsection{Discrete solutions}\label{sssc:disc}

As our potential function, we always consider a lower semi-continuous function
$\phi:X \lra (-\infty,\infty]$ such that
\[ D(\phi):=X \setminus \phi^{-1}(\infty) \neq \emptyset. \]
Given $x \in X$ and $\tau>0$, we define the \emph{Moreau--Yosida approximation}:
\[ \phi_{\tau}(x):=\inf_{z \in X} \left\{ \phi(z) +\frac{d^2(x,z)}{2\tau} \right\} \]
and set
\[ J_{\tau}^{\phi}(x) :=\left\{ z \in X \,\bigg|\, \phi(z) +\frac{d^2(x,z)}{2\tau} =\phi_{\tau}(x) \right\}. \]
For $x \in D(\phi)$ and $z \in J_{\tau}^{\phi}(x)$ (if $J_{\tau}^{\phi}(x) \neq \emptyset$),
it is straightforward from
\[ \phi(z)+\frac{d^2(x,z)}{2\tau} \le \phi(x) \]
that $\phi(z) \le \phi(x)$ and $d^2(x,z) \le 2\tau\{ \phi(x)-\phi(z) \}$.
We consider two kinds of conditions on $\phi$.

\begin{assumption}\label{as:phi}
\begin{enumerate}[(1)]
\item\label{it:coer}
There exists $\tau_*(\phi) \in (0,\infty]$ such that
$\phi_{\tau}(x)>-\infty$ and $J_{\tau}^{\phi}(x) \neq \emptyset$
for all $x \in X$ and $\tau \in (0,\tau_*(\phi))$ (\emph{coercivity}).

\item\label{it:cpt}
For any $Q \in \R$,
bounded subsets of the sub-level set $\{ x \in X \,|\, \phi(x) \le Q \}$ are relatively compact in $X$
(\emph{compactness}).
\end{enumerate}
\end{assumption}

We remark that, if $\phi_{\tau_*}(x_*)>-\infty$ for some $x_* \in X$ and $\tau_*>0$,
then $\phi_{\tau}(x)>-\infty$ for every $x \in X$ and $\tau \in (0,\tau_*)$
(see \cite[Lemma~2.2.1]{AGSbook}).
Then, if the compactness \eqref{it:cpt} holds, we have $J_{\tau}^{\phi}(x) \neq \emptyset$
by the lower semi-continuity of $\phi$ (see \cite[Corollary~2.2.2]{AGSbook}).

\begin{remark}\label{rm:bdd}
If $\diam X<\infty$ and the compactness \eqref{it:cpt} holds, then the lower semi-continuity of $\phi$
implies that every sub-level set $\{ x \in X \,|\, \phi(x) \le Q \}$ is (empty or) compact.
Thus $\phi$ is bounded below and we can take $\tau_*(\phi)=\infty$.
\end{remark}

To construct discrete approximations of gradient curves of $\phi$,
we consider a \emph{partition} of the interval $[0,\infty)$:
\[ \mathscr{P}_{\bm{\tau}}=\{0=t^0_{\bm{\tau}} <t^1_{\bm{\tau}} <\cdots \},
 \qquad \lim_{k \to \infty} t^k_{\bm{\tau}} =\infty, \]
and set
\[ \tau_k:=t^k_{\bm{\tau}}-t^{k-1}_{\bm{\tau}} \quad \text{for }k \in \N,\qquad
 |\bm{\tau}|:=\sup_{k \in \N} \tau_k. \]
We will always assume $|\bm{\tau}|<\tau_*(\phi)$.
Given an initial point $x_0 \in D(\phi)$,
\begin{equation}\label{eq:dgf}
\text{$x_{\bm{\tau}}^0:=x_0$
and recursively choose arbitrary $x_{\bm{\tau}}^k \in J_{\tau_k}^{\phi}(x_{\bm{\tau}}^{k-1})$
for each $k \in \N$.}
\end{equation}
We call $\{x_{\bm{\tau}}^k\}_{k \ge 0}$ a \emph{discrete solution}
of the variational scheme \eqref{eq:dgf} associated with the partition $\mathscr{P}_{\bm{\tau}}$,
which is thought of as a \emph{discrete-time gradient curve} for the potential function $\phi$.
The following a priori estimates (see \cite[Lemma~3.2.2]{AGSbook})
will be useful in the sequel.
We remark that these estimates are easily obtained if $\phi$ is bounded below.

\begin{lemma}[A priori estimates]\label{lm:L322}
Let $\phi:X \lra (-\infty,\infty]$ satisfy Assumption~$\ref{as:phi}$\eqref{it:coer}.
Then, for any $x_* \in X$ and $Q,T>0$, there exists a constant
$C=C(x_*,\tau_*(\phi),Q,T)>0$ such that,
if a partition $\mathscr{P}_{\bm{\tau}}$ and an associated discrete solution
$\{x_{\bm{\tau}}^k\}_{k \ge 0}$ of \eqref{eq:dgf} satisfy
\[ \phi(x_0) \le Q,\quad d^2(x_0,x_*) \le Q,\quad t_{\bm{\tau}}^N \le T,\quad
 |\bm{\tau}| \le \frac{\tau_*(\phi)}{8}, \]
then we have for any $1 \le k \le N$
\[ d^2(x_{\bm{\tau}}^k,x_*) \le C, \qquad
 \sum_{l=1}^k \frac{d^2(x_{\bm{\tau}}^{l-1},x_{\bm{\tau}}^l)}{2\tau_l}
 \le \phi(x_0)-\phi(x_{\bm{\tau}}^k) \le C. \]
\end{lemma}

In particular, for all $1 \le k \le N$,
we have $d^2(x_{\bm{\tau}}^{k-1},x_{\bm{\tau}}^k) \le 2C \tau_k$ and
\begin{equation}\label{eq:apri}
d^2(x_0,x_{\bm{\tau}}^k) \le \left( \sum_{l=1}^k d(x_{\bm{\tau}}^{l-1},x_{\bm{\tau}}^l) \right)^2
 \le \sum_{l=1}^k \frac{d^2(x_{\bm{\tau}}^{l-1},x_{\bm{\tau}}^l)}{\tau_l} \cdot \sum_{l=1}^k \tau_l
 \le 2C t^k_{\bm{\tau}}.
\end{equation}

\subsubsection{Convergence of discrete solutions}\label{sssc:conv}

From here on, let $\phi:(-\infty,\infty] \lra X$ be \emph{$\lambda$-convex}
(also called \emph{$\lambda$-geodesically convex}) for some $\lambda \in \R$ in the sense that
\begin{equation}\label{eq:l-conv}
\phi\big( \gamma(t) \big) \le (1-t)\phi(x) +t\phi(y) -\frac{\lambda}{2}(1-t)td^2(x,y)
\end{equation}
for any $x,y \in D(\phi)$ and some minimal geodesic $\gamma:[0,1] \lra X$ from $x$ to $y$.
(The existence of a minimal geodesic between two points in $D(\phi)$ is included in the definition,
thus in particular $(D(\phi),d)$ is geodesic.)
We remark that the compactness \eqref{it:cpt} in Assumption~\ref{as:phi}
implies the coercivity \eqref{it:coer} in this case
(see \cite[Lemma~2.4.8]{AGSbook}; we even have $\tau_*(\phi)=\infty$ if $\lambda \ge 0$).

Fix an initial point $x_0 \in D(\phi)$.
Take a sequence of partitions
$\{\mathscr{P}_{\bm{\tau}_i}\}_{i \in \N}$ such that $\lim_{i \to \infty} |\bm{\tau}_i|=0$
and associated discrete solutions $\{x_{\bm{\tau}_i}^k\}_{k \ge 0}$
with $x_{\bm{\tau}_i}^0=x_0$.
Under Assumption~\ref{as:phi}\eqref{it:cpt},
by the compactness argument (\cite[Proposition~2.2.3]{AGSbook}),
a subsequence of the interpolated curves
\begin{equation}\label{eq:barx}
\bm{\bar{x}}_{\bm{\tau}_i}(0) :=x_0,\qquad
 \bm{\bar{x}}_{\bm{\tau}_i}(t) :=x_{\bm{\tau}_i}^k\quad
 \text{for } t \in (t_{\bm{\tau}_i}^{k-1},t_{\bm{\tau}_i}^k]
\end{equation}
converges to a curve $\xi:[0,\infty) \lra D(\phi)$ point-wise in $t \in [0,\infty)$.
In general, under the coercivity and $\lambda$-convexity of $\phi$ (but without the compactness),
if a curve $\xi$ is obtained as above
(called a \emph{generalized minimizing movement}; see \cite[Definition~2.0.6]{AGSbook}),
then it is locally Lipschitz on $(0,\infty)$ and satisfies $\lim_{t \downarrow 0}\xi(t)=x_0$
as well as the \emph{energy dissipation identity}:
\begin{equation}\label{eq:gf}
\phi\big( \xi(T) \big) =\phi\big( \xi(S) \big)
 -\frac{1}{2} \int_S^T \{ |\dot{\xi}|^2 +|\nabla \phi|^2(\xi) \} \,dt.
\end{equation}
Here
\[ |\dot{\xi}|(t) :=\lim_{s \to t} \frac{d(\xi(s),\xi(t))}{|t-s|} \]
is the \emph{metric speed} existing at almost all $t$, and
\[ |\nabla \phi|(x) :=\limsup_{y \to x} \frac{\max\{\phi(x)-\phi(y),0\}}{d(x,y)} \]
is the (descending) \emph{local slope}
(see \cite[Theorem~2.4.15]{AGSbook}).
We remark that $|\nabla\phi|$ is lower semi-continuous (\cite[Corollary~2.4.10]{AGSbook}) and
$ \lim_{i \to \infty}\phi(\bm{\bar{x}}_{\bm{\tau}_i}(t)) =\phi(\xi(t))$
for all $t \ge 0$ (\cite[Theorem~2.3.3]{AGSbook}).
The equation \eqref{eq:gf} can be thought of as a metric formulation of the differential equation
$\dot{\xi}(t)=-\nabla\phi(\xi(t))$, thus $\xi$ will be called a \emph{gradient curve} of $\phi$
starting from $x_0$.
We remark that one does not have uniqueness of gradient curves in this generality
(see \cite[Example~4.23]{AG} for a simple example in the $\ell^2_{\infty}$-space).

\begin{remark}\label{rm:x0}
One can relax the scheme by allowing varying initial points: $x_{\bm{\tau}_i}^0 \neq x_0$.
Then assuming $x_{\bm{\tau}_i}^0 \to x_0$ and $\phi(x_{\bm{\tau}_i}^0) \to \phi(x_0)$
yields the same convergence results.
In our setting, such a convergence can also follow from the comparison estimate \eqref{eq:t/s-2}
(see the proof of Theorem~\ref{th:T422}).
\end{remark}

\subsection{$\CAT(1)$-spaces}\label{ssc:CAT}

We refer to \cite{BBI} for the basics of $\CAT(1)$-spaces and for more general metric geometry.

Given three points $x,y,z\in X$ with $d(x,y)+d(y,z)+d(z,x)<2\pi$,
we can take corresponding points $\tilde{x},\tilde{y},\tilde{z}$
in the $2$-dimensional unit sphere $\Sph^2$ (uniquely up to rigid motions) such that
\[ d_{\Sph^2}(\tilde{x},\tilde{y})=d(x,y), \qquad d_{\Sph^2}(\tilde{y},\tilde{z})=d(y,z), \qquad
 d_{\Sph^2}(\tilde{z},\tilde{x})=d(z,x). \]
We call $\triangle\tilde{x}\tilde{y}\tilde{z}$ a \emph{comparison triangle} of $\triangle xyz$ in $\Sph^2$.

\begin{definition}[$\CAT(1)$-spaces]\label{df:CAT}
A geodesic metric space $(X,d)$ is called a \emph{$\CAT(1)$-space} if,
for any $x,y,z \in X$ with $d(x,y)+d(y,z)+d(z,x)<2\pi$ and any minimal geodesic
$\gamma:[0,1] \lra X$ from $y$ to $z$, we have
\[ d\big( x,\gamma(t) \big) \le d_{\Sph^2}\big( \tilde{x},\tilde{\gamma}(t) \big) \]
at all $t \in [0,1]$, where 
$\triangle\tilde{x}\tilde{y}\tilde{z} \subset \Sph^2$ is a comparison triangle of $\triangle xyz$
and $\tilde{\gamma}:[0,1] \lra \Sph^2$ is the minimal geodesic
from $\tilde{y}$ to $\tilde{z}$.
\end{definition}

It is readily observed from the definition that each pair of points $x,y \in X$ with $d(x,y)<\pi$
is joined by a unique minimal geodesic.
Fundamental examples of $\CAT(1)$-spaces are the following.

\begin{example}\label{ex:CAT}
(1)
A complete, simply connected Riemannian manifold endowed with the Riemannian distance
is a $\CAT(1)$-space if and only if its sectional curvature is not greater than $1$ everywhere.

(2)
All \emph{$\CAT(0)$-spaces}, which are defined similarly to Definition~\ref{df:CAT}
by replacing $\Sph^2$ with $\R^2$, are $\CAT(1)$-spaces.
Important examples of $\CAT(0)$-spaces are Hadamard manifolds,
Hilbert spaces, trees and Euclidean buildings.

(3)
Further examples of $\CAT(1)$-spaces include orbifolds obtained as quotient spaces of $\CAT(1)$-manifolds,
and spherical buildings.
See \cite[\S 9.1]{BBI} for more examples.
\end{example}

\begin{remark}\label{rm:CAT}
For general $\kappa \in \R$, $\CAT(\kappa)$-spaces are defined in the same manner
by employing comparison triangles in the $2$-dimensional space form of constant curvature $\kappa$.
If $(X,d)$ is a $\CAT(\kappa)$-space, then it is also $\CAT(\kappa')$ for all $\kappa'>\kappa$
and the scaled metric space $(X,cd)$ with $c>0$ is $\CAT(c^{-2}\kappa)$.
Therefore considering $\CAT(1)$-spaces covers all $\CAT(\kappa)$-spaces up to scaling.
\end{remark}

As was mentioned in the introduction,
the properties of $\CAT(1)$-spaces needed in our discussion are
only the semi-convexity of the squared distance function and the commutativity \eqref{eq:commu}.
The latter is a consequence of the first variation formula.
Let us review them.

\begin{lemma}[Semi-convexity of distance functions]\label{lm:Kconv}
Let $(X,d)$ be a $\CAT(1)$-space and take $R \in (0,\pi)$.
Then there exists $K=K(R) \in \R$ such that the squared distance function $d^2(x,\cdot)$ is $K$-convex
on the open $R$-ball $B(x,R)$ for all $x \in X$.
\end{lemma}

\begin{proof}
By the definition of $\CAT(1)$-spaces, it is enough to show the claim in $\Sph^2$.
Then the $K$-convexity is a direct consequence of the smoothness of $d_{\Sph^2}^2(\tilde{x},\cdot)$
on $B(\tilde{x},\pi)$.
$\qedd$
\end{proof}

Clearly $K(R)>0$ if $R<\pi/2$ (see, e.g., \cite{Oconv} for the precise estimate)
and $K(R)<0$ if $R>\pi/2$.
We can define the \emph{angle} between two geodesics
$\gamma$ and $\eta$ emanating from the same point $\gamma(0)=\eta(0)=x$ by
\[ \angle_x(\gamma,\eta):=\lim_{s,t \downarrow 0}\angle\widetilde{\gamma(s)}\tilde{x}\widetilde{\eta(t)}, \]
where $\angle\widetilde{\gamma(s)}\tilde{x}\widetilde{\eta(t)}$ is the angle at $\tilde{x}$ of a comparison triangle
$\triangle \widetilde{\gamma(s)}\tilde{x}\widetilde{\eta(t)}$ in $\Sph^2$.
By the definition of the angle, we obtain the following (see \cite[Theorem~4.5.6, Remark~4.5.12]{BBI}).

\begin{theorem}[First variation formula]\label{th:1vf}
Let $\gamma:[0,1] \lra X$ be a geodesic from $x$ to $z$, and take $y \in X$ with $0<d(x,y)<\pi$.
Then we have
\[ \lim_{s \downarrow 0}\frac{d(\gamma(s),y)-d(x,y)}{s}
 =-d(x,z) \cos \angle_x (\gamma,\eta), \]
where $\eta:[0,1] \to X$ is the unique minimal geodesic from $x$ to $y$.
\end{theorem}

\section{Key lemma}\label{sc:key}

In this section, let $(X,d)$ be a complete $\CAT(1)$-space and
$\phi:X \lra (-\infty,\infty]$ satisfy the $\lambda$-convexity for some $\lambda \in \R$ and
Assumption~\ref{as:phi}\eqref{it:coer}.
Note that \eqref{eq:l-conv} holds along every minimal geodesic
since minimal geodesics are unique between points of distance $<\pi$.
The next lemma, generalizing \cite[Theorem~4.1.2(ii)]{AGSbook}
to the case where both $\lambda$ and $K$ can be negative,
will be a key tool in the following sections.

\begin{lemma}[Key lemma]\label{lm:key}
Let $x \in D(\phi)$ and $\tau \in (0,\min\{\pi^2/(2C),\tau_*(\phi)/8\})$
with $C=C(x,\tau_*(\phi),\phi(x),\tau_*(\phi)/8)$ from Lemma~$\ref{lm:L322}$.
Take $x_{\tau} \in J^{\phi}_{\tau}(x)$.
Then we have, for any $y \in D(\phi) \cap B(x_{\tau},R-d(x,x_{\tau}))$ with $R<\pi$
and for $K=K(R)$ as in Lemma~$\ref{lm:Kconv}$,
\begin{align*}
d^2(x_{\tau},y)
&\le d^2(x,y) -\lambda \tau d^2(x_{\tau},y) +2\tau\{ \phi(y)-\phi(x_{\tau}) \}
 -\frac{K}{2}d^2(x,x_{\tau}) \\
&\le d^2(x,y) -\lambda \tau d^2(x_{\tau},y) +2\tau\{ \phi(y)-\phi(x_{\tau}) \} \\
&\quad +\max\{0,-K\} \cdot \tau\{ \phi(x)-\phi(x_{\tau}) \}.
\end{align*}
\end{lemma}

(We remark that $C=C(x,\tau_*(\phi),\phi(x),\tau_*(\phi)/8)$ in the lemma
means the constant $C(x,\tau_*(\phi),Q,T)$ from Lemma~\ref{lm:L322}
with $Q=\phi(x)$ and $T=\tau_*(\phi)/8$.)

\begin{proof}
Observe that $d^2(x,x_{\tau}) \le 2C\tau <\pi^2$ by Lemma~\ref{lm:L322}
and the choice of $\tau$.
Let $\gamma:[0,1] \lra X$ be the minimal geodesic from $x_{\tau}$ to $y$,
and $\eta:[0,1] \lra X$ from $x_{\tau}$ to $x$.
For any $s \in (0,1)$, by the definition of $J^{\phi}_{\tau}(x)$ and the $\lambda$-convexity of $\phi$,
we have
\begin{align*}
\phi(x_{\tau}) +\frac{d^2(x,x_{\tau})}{2\tau}
&\le \phi\big( \gamma(s) \big) +\frac{d^2(x,\gamma(s))}{2\tau} \\
&\le (1-s)\phi(x_{\tau}) +s\phi(y) -\frac{\lambda}{2}(1-s)sd^2(x_{\tau},y)
 +\frac{d^2(x,\gamma(s))}{2\tau}.
\end{align*}
Hence
\[ \phi(x_{\tau})
 \le \phi(y) +\frac{1}{2\tau}\frac{d^2(x,\gamma(s))-d^2(x,x_{\tau})}{s}
 -\frac{\lambda}{2}(1-s)d^2(x_{\tau},y). \]
Applying the first variation formula (Theorem~\ref{th:1vf}) twice,
we observe the \emph{commutativity}:
\begin{align}
\lim_{s \downarrow 0} \frac{d^2(x,\gamma(s))-d^2(x,x_{\tau})}{s}
&= -2d(x_{\tau},x) d(x_{\tau},y) \cos\angle_{x_{\tau}}(\gamma,\eta) \nonumber\\
&= \lim_{t \downarrow 0} \frac{d^2(\eta(t),y)-d^2(x_{\tau},y)}{t}. \label{eq:comm}
\end{align}
Notice that $\eta$ is contained in $B(y,R)$ by the choice of $y$.
Thus it follows from the $K$-convexity of $d^2(\cdot,y)$ in $B(y,R)$ that
\begin{align*}
\lim_{t \downarrow 0} \frac{d^2(\eta(t),y)-d^2(x_{\tau},y)}{t}
&\le d^2(x,y) -d^2(x_{\tau},y) -\frac{K}{2}d^2(x,x_{\tau}) \\
&\le d^2(x,y) -d^2(x_{\tau},y) +\max\{0,-K\} \cdot \tau \{\phi(x)-\phi(x_{\tau})\}.
\end{align*}
This completes the proof.
$\qedd$
\end{proof}

\begin{remark}\label{rm:key}
(a)
Used in the proof of \cite[Theorem~4.1.2(ii)]{AGSbook} is the direct application of the convexity
of $\phi$ and $d^2(x,\cdot)$ along $\gamma$, which implies in our setting
\[ \frac{K}{2} d^2(x_{\tau},y)
 \le d^2(x,y) -\lambda \tau d^2(x_{\tau},y) +2\tau\{ \phi(y)-\phi(x_{\tau}) \} -d^2(x,x_{\tau}). \]
This coincides with our estimate when $K=2$.
The commutativity was used to move the coefficient $K/2$ from $d^2(x_{\tau},y)$ to $d^2(x,x_{\tau})$,
then one can efficiently estimate $d^2(x_{\tau},y)-d^2(x,y)$ (see Theorem~\ref{th:T414}).

(b)
As we mentioned in the introduction (see the paragraph including \eqref{eq:commu}),
the Riemannian nature of the space (i.e., the angle) is essential in the commutativity \eqref{eq:comm}.
In fact, on a Finsler manifold $(M,F)$, \eqref{eq:commu} (written using only the distance) implies
\[ g_v(v,w)=g_w(v,w) \qquad \text{for all}\ v,w \in T_xM \setminus \{0\},\ x \in M. \]
These notations and the basics of Finsler geometry can be found in \cite{OSnc} for instance.
Thus we find, for $v \neq \pm w$,
\begin{align*}
&F^2(v+w)+F^2(v-w)
 = g_{v+w}(v+w,v+w) +g_{v-w}(v-w,v-w) \\
&= g_{v+w}(v,v+w) +g_{v+w}(w,v+w) +g_{v-w}(v,v-w) -g_{v-w}(w,v-w) \\
&= g_v(v,v+w) +g_w(w,v+w) +g_v(v,v-w) -g_w(w,v-w) \\
&= 2g_v(v,v) +2g_w(w,w) =2F^2(v) +2F^2(w).
\end{align*}
This is the parallelogram identity on $T_xM$ and hence $F$ is Riemannian.

The commutativity \eqref{eq:commu} is the essential property connecting
the (geodesic) convexity of a function and the contraction property
of its gradient flow (see Theorem~\ref{th:P431}).
On Finsler manifolds, the contraction property is characterized by
the \emph{skew-convexity} which is different from the usual convexity along geodesics
(see \cite{OSnc} for details).
\end{remark}

\section{Applications to gradient flows}\label{sc:gf}

The estimate in Lemma~\ref{lm:key} is worse than the one in \cite[Theorem~4.1.2(ii)]{AGSbook}
because of the generality that $K$ can be less than $2$ and even negative.
Nonetheless, as we shall see in this section,
Lemma~\ref{lm:key} is enough to generalize the argumentation in Chapter~4 of \cite{AGSbook}.
We will give at least sketches of the proofs for completeness.

Our argument covers two cases.
In both cases, $(X,d)$ is complete,
$\phi:X \lra (-\infty,\infty]$ is lower semi-continuous, $\lambda$-convex and $D(\phi) \neq \emptyset$.

\renewcommand{\thecase}{\Roman{case}}
\begin{case}\label{case:CAT}
$(X,d)$ is a $\CAT(1)$-space.
\end{case}

\begin{case}\label{case:Kcon}
$(X,d)$ satisfies the commutativity~\eqref{eq:comm} and the $K$-convexity of the squared distance function,
and $\phi$ satisfies the coercivity condition $($Assumption~$\ref{as:phi}\eqref{it:coer})$.
$($To be precise, the commutativity, $K$-convexity and $\lambda$-convexity
are assumed to hold along the same family of geodesics.$)$
\end{case}

We stress that both $\lambda,K \in \R$ can be negative.
In Case~\ref{case:Kcon}, the $K$-convexity is assumed globally,
thus the assertion of Lemma~\ref{lm:key} holds for any $x,y \in D(\phi)$ and $\tau \in (0,\tau_*(\phi))$.
We recall that the coercivity holds if, for instance,
the compactness condition (Assumption~\ref{as:phi}\eqref{it:cpt}) is satisfied.
In Case~\ref{case:CAT},
the coercivity is guaranteed by restricting ourselves to balls with radii $\le R<\pi/4$.
In these (convex) balls the squared distance function is $K$-convex with $K=K(2R)>0$ from Lemma~\ref{lm:Kconv},
and hence $\phi+d^2(x,\cdot)/(2\tau)$ is $(\lambda+K/(2\tau))$-convex.
This implies that $J^{\phi}_{\tau}(x)$ is nonempty and consists of a single point
for $\tau \in (0,-K/(2\lambda))$ even if $\lambda<0$ (by, for example, \cite[Lemma~2.4.8]{AGSbook}).
By the same reasoning, if $K>0$ in Case~\ref{case:Kcon},
then Assumption~\ref{as:phi}\eqref{it:coer} is redundant.

To include both cases keeping clarity of the presentation,
we discuss under the global $K$-convexity of the squared distance function
and Assumption~\ref{as:phi}\eqref{it:coer}.
Thus we implicitly assume $\diam X<\pi/2$ if we are in Case~\ref{case:CAT}.
This costs no generality for the construction of gradient curves since it is a local problem.
We explain how to extend the properties of the gradient flow
to the case of $\diam X \ge \pi/2$ in \S \ref{ssc:pi}.

\subsection{Interpolations}\label{ssc:piece}

Given an initial point $x_0 \in D(\phi)$ and a partition $\mathscr{P}_{\bm{\tau}}$
with $|\bm{\tau}|<\tau_*(\phi)$, we fix a discrete solution $\{x_{\bm{\tau}}^k\}_{k \ge 0}$
of \eqref{eq:dgf}.
Let us also take a point $y \in X$.
Similarly to Chapter~4 of \cite{AGSbook}, we interpolate the discrete data
$x_{\bm{\tau}}^k$, $d(x_{\bm{\tau}}^k,y)$ and $\phi(x_{\bm{\tau}}^k)$ as follows
(recall \eqref{eq:barx}):

For $t \in (t_{\bm{\tau}}^{k-1},t_{\bm{\tau}}^k]$, $k \in \N$, define
\begin{align*}
\bm{\bar{x}_{\tau}}(t) &:= x_{\bm{\tau}}^k
 \in J^{\phi}_{\tau_k}(x_{\bm{\tau}}^{k-1}) \quad (\bm{\bar{x}_{\tau}}(0):=x_0), \\
\bm{\bar{d}_{\tau}}(t;y)
&:= \left\{ d^2(x_{\bm{\tau}}^{k-1},y) +\frac{t-t_{\bm{\tau}}^{k-1}}{\tau_k}
 \{ d^2(x_{\bm{\tau}}^k,y)-d^2(x_{\bm{\tau}}^{k-1},y) \} \right\}^{1/2}, \\
\bm{\bar{\phi}_{\tau}}(t)
&:= \phi(x_{\bm{\tau}}^{k-1}) +\frac{t-t_{\bm{\tau}}^{k-1}}{\tau_k}
 \{ \phi(x_{\bm{\tau}}^k)-\phi(x_{\bm{\tau}}^{k-1}) \}.
\end{align*}
Recall that $\tau_k=t_{\bm{\tau}}^k -t_{\bm{\tau}}^{k-1}$ and
note that $\bm{\bar{\phi}_{\tau}}$ is non-increasing.

Then Lemma~\ref{lm:key} yields the following discrete version of the \emph{evolution variational inequality}
(see \cite[Theorem~4.1.4]{AGSbook}; we remark that our residual function
$\mathscr{R}_{\bm{\tau},K}$ is different from that in \cite{AGSbook} and depends on $K$).

\begin{theorem}[Discrete evolution variational inequality]\label{th:T414}
Assuming $|\bm{\tau}|<\tau_*(\phi)$, we have
\[ \frac{1}{2} \frac{d}{dt} \big[ \bm{\bar{d}_{\tau}}^2(t;y) \big]
 +\frac{\lambda}{2}d^2\big( \bm{\bar{x}_{\tau}}(t),y \big) +\bm{\bar{\phi}_{\tau}}(t) -\phi(y)
 \le \mathscr{R}_{\bm{\tau},K}(t) \]
for almost all $t \in (0,T)$ and all $y \in D(\phi)$,
where for $t \in (t_{\bm{\tau}}^{k-1},t_{\bm{\tau}}^k]$
\[ \mathscr{R}_{\bm{\tau},K}(t):=
 \left( \frac{t_{\bm{\tau}}^k-t}{\tau_k} +\frac{\max\{0,-K\}}{2} \right)
 \{ \phi(x_{\bm{\tau}}^{k-1})-\phi(x_{\bm{\tau}}^k) \}. \]
\end{theorem}

\begin{proof}
Note first that Lemma~\ref{lm:key} is available merely under $|\bm{\tau}|<\tau_*(\phi)$
in the current situation.
Then we immediately obtain for $t \in (t_{\bm{\tau}}^{k-1},t_{\bm{\tau}}^k)$
\begin{align*}
&\frac{1}{2} \frac{d}{dt} \big[ \bm{\bar{d}_{\tau}}^2(t;y) \big]
 = \frac{d^2(x_{\bm{\tau}}^k,y)-d^2(x_{\bm{\tau}}^{k-1},y)}{2\tau_k} \\
&\le -\frac{\lambda}{2} d^2(x_{\bm{\tau}}^k,y) +\phi(y) -\phi(x_{\bm{\tau}}^k)
 +\frac{\max\{0,-K\}}{2} \{ \phi(x_{\bm{\tau}}^{k-1})-\phi(x_{\bm{\tau}}^k) \} \\
&= -\frac{\lambda}{2}d^2\big( \bm{\bar{x}_{\tau}}(t),y \big)
 +\phi(y) -\bm{\bar{\phi}_{\tau}}(t)
 +\frac{t_{\bm{\tau}}^k-t}{\tau_k} \{ \phi(x_{\bm{\tau}}^{k-1})-\phi(x_{\bm{\tau}}^k) \} \\
&\quad +\frac{\max\{0,-K\}}{2} \{ \phi(x_{\bm{\tau}}^{k-1})-\phi(x_{\bm{\tau}}^k) \}.
\end{align*}
This completes the proof.
$\qedd$
\end{proof}

Taking the limit in Theorem~\ref{th:T414} as $|\bm{\tau}| \to 0$ will indeed lead to
the (continuous) evolution variational inequality (Theorem~\ref{th:T432}).
Another application of Theorem~\ref{th:T414} is a comparison between two discrete solutions
generated from different partitions:
\[ \mathscr{P}_{\bm{\tau}}=\{0=t^0_{\bm{\tau}} <t^1_{\bm{\tau}} <\cdots \}, \qquad
 \mathscr{P}_{\bm{\sigma}}=\{0=s^0_{\bm{\sigma}} <s^1_{\bm{\sigma}} <\cdots \}. \]
We set $\sigma_l:=s^l_{\bm{\sigma}}-s^{l-1}_{\bm{\sigma}}$ for $l \in \N$.
We first observe the following modification of Theorem~\ref{th:T414} (see \cite[Lemma~4.1.6]{AGSbook}).

\begin{lemma}\label{lm:L416}
Suppose $|\bm{\tau}|<\tau_*(\phi)$ and $\lambda \le 0$.
Then we have
\[ \frac{1}{2} \frac{d}{dt} \big[ \bm{\bar{d}_{\tau}}^2(t;y) \big]
 +\frac{\lambda}{2}\bm{\bar{d}_{\tau}}^2(t;y) +\lambda \mathscr{D}_{\bm{\tau}}(t) \bm{\bar{d}_{\tau}}(t;y)
 +\bm{\bar{\phi}_{\tau}}(t) -\phi(y)
 \le \mathscr{R}_{\bm{\tau},K}(t) -\frac{\lambda}{2} \mathscr{D}_{\bm{\tau}}^2(t) \]
for almost all $t \in (0,T)$ and all $y \in D(\phi)$,
where for $t \in (t_{\bm{\tau}}^{k-1},t_{\bm{\tau}}^k]$
\[ \mathscr{D}_{\bm{\tau}}(t):=
 \frac{t_{\bm{\tau}}^k -t}{\tau_k} d(x_{\bm{\tau}}^{k-1},x_{\bm{\tau}}^k). \]
\end{lemma}

\begin{proof}
This is a consequence of Theorem~\ref{th:T414} and the inequality
\begin{align*}
d\big( \bm{\bar{x}_{\tau}}(t),y \big) &=d(x_{\bm{\tau}}^k,y)
 \le \frac{t-t_{\bm{\tau}}^{k-1}}{\tau_k} d(x_{\bm{\tau}}^k,y)
 +\frac{t_{\bm{\tau}}^k -t}{\tau_k} \{ d(x_{\bm{\tau}}^{k-1},y)+d(x_{\bm{\tau}}^{k-1},x_{\bm{\tau}}^k) \} \\
&\le \bm{\bar{d}_{\tau}}(t;y) +\mathscr{D}_{\bm{\tau}}(t),
\end{align*}
which follows only from the triangle inequality and the convexity of $f(r)=r^2$, $r \in \R$.
$\qedd$
\end{proof}

Notice that Lemma~\ref{lm:L416} reduces to Theorem~\ref{th:T414} when $\lambda=0$.
Applying a version of the Gronwall lemma, we obtain from Lemma~\ref{lm:L416}
(or directly from Theorem~\ref{th:T414} if $\lambda=0$) a comparison estimate between
$\{x_{\bm{\tau}}^k\}_{k \ge 0}$ and $\{y_{\bm{\sigma}}^l\}_{l \ge 0}$
with $y_{\bm{\sigma}}^0=y_0 \in D(\phi)$ (see \cite[Corollaries~4.1.5, 4.1.7]{AGSbook}).
For $s \in (s_{\bm{\sigma}}^{l-1},s_{\bm{\sigma}}^l]$, we set
\[ \bm{\bar{d}_{\tau\sigma}}(t,s):=
 \left\{ \bm{\bar{d}_{\tau}}^2(t;y_{\bm{\sigma}}^{l-1}) +\frac{s-s_{\bm{\sigma}}^{l-1}}{\sigma_l}
 \{ \bm{\bar{d}_{\tau}}^2(t;y_{\bm{\sigma}}^l) -\bm{\bar{d}_{\tau}}^2(t;y_{\bm{\sigma}}^{l-1}) \} \right\}^{1/2}. \]
Observe that,
for $(t,s) \in (t_{\bm{\tau}}^{k-1},t_{\bm{\tau}}^k] \times (s_{\bm{\sigma}}^{l-1},s_{\bm{\sigma}}^l]$,
\begin{align*}
\bm{\bar{d}}^2_{\bm{\tau\sigma}}(t,s) &=
 \left( 1-\frac{t-t_{\bm{\tau}}^{k-1}}{\tau_k} \right) \left( 1-\frac{s-s_{\bm{\sigma}}^{l-1}}{\sigma_l} \right)
 d^2(x_{\bm{\tau}}^{k-1},y_{\bm{\sigma}}^{l-1}) \\
&\quad +\frac{t-t_{\bm{\tau}}^{k-1}}{\tau_k} \left( 1-\frac{s-s_{\bm{\sigma}}^{l-1}}{\sigma_l} \right)
 d^2(x_{\bm{\tau}}^k,y_{\bm{\sigma}}^{l-1}) \\
&\quad +\left( 1-\frac{t-t_{\bm{\tau}}^{k-1}}{\tau_k} \right) \frac{s-s_{\bm{\sigma}}^{l-1}}{\sigma_l}
 d^2(x_{\bm{\tau}}^{k-1},y_{\bm{\sigma}}^l)
 +\frac{t-t_{\bm{\tau}}^{k-1}}{\tau_k} \frac{s-s_{\bm{\sigma}}^{l-1}}{\sigma_l}
 d^2(x_{\bm{\tau}}^k,y_{\bm{\sigma}}^l).
\end{align*}

\begin{corollary}[Comparison between two discrete solutions]\label{cr:C417}
Assume $\lambda \le 0$ and $|\bm{\tau}|,|\bm{\sigma}|<\tau_*(\phi)$.
Then we have, for almost all $t>0$,
\begin{align}
\frac{d}{dt} \big[ \bm{\bar{d}_{\tau\sigma}}^2(t,t) \big] +2\lambda \bm{\bar{d}_{\tau\sigma}}^2(t,t)
&\le -2\lambda \big( \mathscr{D}_{\bm{\tau}}(t)+\mathscr{D}_{\bm{\sigma}}(t) \big) \bm{\bar{d}_{\tau\sigma}}(t,t)
 \nonumber\\
&\quad +2\big( \mathscr{R}_{\bm{\tau},K}(t) +\mathscr{R}_{\bm{\sigma},K}(t) \big)
 -\lambda \big( \mathscr{D}_{\bm{\tau}}^2(t) +\mathscr{D}_{\bm{\sigma}}^2(t) \big). \label{eq:t/s-1}
\end{align}
Moreover, for all $T>0$,
\begin{align}
e^{\lambda T} \bm{\bar{d}_{\tau\sigma}}(T,T)
&\le \left\{ d^2(x_0,y_0) +\int_0^T\! e^{2\lambda t} \big\{
 2\big( \mathscr{R}_{\bm{\tau},K}(t) +\mathscr{R}_{\bm{\sigma},K}(t) \big)
 -\lambda \big( \mathscr{D}_{\bm{\tau}}^2(t) +\mathscr{D}_{\bm{\sigma}}^2(t) \big) \big\} \,dt \right\}^{\!1/2}
 \nonumber\\
&\quad -2\lambda \int_0^T e^{\lambda t}
 \big( \mathscr{D}_{\bm{\tau}}(t) +\mathscr{D}_{\bm{\sigma}}(t) \big) \,dt. \label{eq:t/s-2}
\end{align}
\end{corollary}

\begin{proof}
For each fixed $s$, it follows from Lemma~\ref{lm:L416} that
\begin{align*}
&\frac{1}{2} \frac{\del}{\del t} \big[ \bm{\bar{d}_{\tau\sigma}}^2(t,s) \big]
 +\frac{\lambda}{2}\bm{\bar{d}_{\tau\sigma}}^2(t,s) +\bm{\bar{\phi}_{\tau}}(t) -\bm{\bar{\phi}_{\sigma}}(s) \\
&\le -\lambda \mathscr{D}_{\bm{\tau}}(t)
 \left\{ \left( 1-\frac{s-s_{\bm{\sigma}}^{l-1}}{\sigma_l} \right) \bm{\bar{d}_{\tau}}(t;y_{\bm{\sigma}}^{l-1})
 +\frac{s-s_{\bm{\sigma}}^{l-1}}{\sigma_l} \bm{\bar{d}_{\tau}}(t;y_{\bm{\sigma}}^l) \right\}
 +\mathscr{R}_{\bm{\tau},K}(t) -\frac{\lambda}{2} \mathscr{D}_{\bm{\tau}}^2(t) \\
&\le -\lambda \mathscr{D}_{\bm{\tau}}(t) \bm{\bar{d}_{\tau\sigma}}(t,s)
 +\mathscr{R}_{\bm{\tau},K}(t) -\frac{\lambda}{2} \mathscr{D}_{\bm{\tau}}^2(t).
\end{align*}
Combining this with a similar inequality
\[ \frac{1}{2} \frac{\del}{\del s} \big[ \bm{\bar{d}_{\tau\sigma}}^2(t,s) \big]
 +\frac{\lambda}{2}\bm{\bar{d}_{\tau\sigma}}^2(t,s) +\bm{\bar{\phi}_{\sigma}}(s) -\bm{\bar{\phi}_{\tau}}(t)
 \le -\lambda \mathscr{D}_{\bm{\sigma}}(s) \bm{\bar{d}_{\tau\sigma}}(t,s)
 +\mathscr{R}_{\bm{\sigma},K}(s) -\frac{\lambda}{2} \mathscr{D}_{\bm{\sigma}}^2(s), \]
we obtain the first inequality \eqref{eq:t/s-1}.
The second assertion \eqref{eq:t/s-2} is a consequence of \eqref{eq:t/s-1}
via a version of the Gronwall lemma (see \cite[Lemma~4.1.8]{AGSbook}).
$\qedd$
\end{proof}

\subsection{Convergence of discrete solutions}\label{ssc:conv}

Corollary~\ref{cr:C417} implies that the discrete solutions
$\{x_{\bm{\tau}}^k\}_{k \ge 0}$ converges to a gradient curve
as $|\bm{\tau}| \to 0$ and the limit curve is independent of
the choice of the discrete solutions (generalizing \cite[Theorem~4.2.2]{AGSbook}).
Recall \S \ref{sssc:conv} for properties of gradient curves.

\begin{theorem}[Unique limits of discrete solutions]\label{th:T422}
Fix an initial point $x_0 \in D(\phi)$ and consider discrete solutions $\{ x_{\bm{\tau}_i}^k \}_{k \ge 0}$
with $x_{\bm{\tau}_i}^0=x_0$ associated with a sequence of partitions
$\{\mathscr{P}_{\bm{\tau}_i}\}_{i \in \N}$ such that $\lim_{i \to \infty}|\bm{\tau}_i|=0$.
Then the interpolated curve $\bm{\bar{x}}_{\bm{\tau}_i}:[0,\infty) \lra X$ as in \S $\ref{ssc:piece}$
converges to a curve $\xi:[0,\infty) \lra X$ with $\xi(0)=x_0$ as $i \to \infty$
uniformly on each bounded interval $[0,T]$.
In particular, the limit curve $\xi$ is independent of the choices
of the sequence of partitions and discrete solutions.
\end{theorem}

\begin{proof}
Since $(X,d)$ is complete,
it is sufficient to show $\bm{\bar{d}_{\tau_i \tau_j}}(t,t) \to 0$ as $i,j \to \infty$ uniformly on $[0,T]$,
in the notations of Corollary~\ref{cr:C417}.
Let $\lambda <0$ without loss of generality.
Take $i_0$ large enough to satisfy $|\bm{\tau}_i| \le \tau_*(\phi)/8$ for all $i \ge i_0$,
and put $K':=\min\{0,K\} \le 0$.

The integral of $\mathscr{R}_{\bm{\tau},K}$ (recall Theorem~\ref{th:T414} for the definition)
is calculated as
\[ \int_{t_{\bm{\tau}}^{k-1}}^{t_{\bm{\tau}}^k} \mathscr{R}_{\bm{\tau},K} \,dt
 =\left( \frac{1}{2}-\frac{K'}{2} \right) \tau_k
 \{ \phi(x_{\bm{\tau}}^{k-1}) -\phi(x_{\bm{\tau}}^k) \}. \]
Similarly, together with the canonical estimate
$d^2(x_{\bm{\tau}}^{k-1},x_{\bm{\tau}}^k)
 \le 2\tau_k \{\phi(x_{\bm{\tau}}^{k-1})-\phi(x_{\bm{\tau}}^k) \}$,
we have
\[ \int_{t_{\bm{\tau}}^{k-1}}^{t_{\bm{\tau}}^k} \mathscr{D}^2_{\bm{\tau}} \,dt
 \le \frac{\tau_k}{3} \cdot
 2\tau_k \{\phi(x_{\bm{\tau}}^{k-1})-\phi(x_{\bm{\tau}}^k) \}. \]
This also implies
\[ \int_0^{t_{\bm{\tau}}^k} e^{\lambda t} \mathscr{D}_{\bm{\tau}}(t) \,dt
 \le \bigg( \int_0^{t_{\bm{\tau}}^k} e^{2\lambda t} \,dt \bigg)^{1/2}
 \bigg( \int_0^{t_{\bm{\tau}}^k} \mathscr{D}^2_{\bm{\tau}} \,dt \bigg)^{1/2}
 \le \sqrt{-\frac{1}{2\lambda}} \sqrt{\frac{2}{3}} |\bm{\tau}| \sqrt{\phi(x_0)-\phi(x_{\bm{\tau}}^k)}. \]
Combining these with \eqref{eq:t/s-2},
we obtain for $i,j \ge i_0$, $t \le t^k_{\bm{\tau}_i} \le T$ and $t \le t^l_{\bm{\tau}_j} \le T$,
\begin{align*}
e^{\lambda t} \bm{\bar{d}_{\tau_i \tau_j}}(t,t)
&\le \left\{ \int_0^t \{ 2(\mathscr{R}_{\bm{\tau}_i,K}+\mathscr{R}_{\bm{\tau}_j,K})
 -\lambda (\mathscr{D}^2_{\bm{\tau}_i}+\mathscr{D}^2_{\bm{\tau}_j}) \} \,ds \right\}^{1/2} \\
&\quad -2\lambda \int_0^t e^{\lambda s} \big( \mathscr{D}_{\bm{\tau}_i}(s)+\mathscr{D}_{\bm{\tau}_j}(s) \big) \,ds \\
&\le \bigg\{ \left( 1-K'-\frac{2\lambda}{3}|\bm{\tau}_i| \right)
 |\bm{\tau}_i| \{ \phi(x_0)-\phi(x_{\bm{\tau}_i}^k) \} \\
&\qquad +\left( 1-K'-\frac{2\lambda}{3}|\bm{\tau}_j| \right)
 |\bm{\tau}_j| \{ \phi(x_0)-\phi(x_{\bm{\tau}_j}^l) \} \bigg\}^{1/2} \\
&\quad +\sqrt{-\frac{4\lambda}{3}} \left( |\bm{\tau}_i| \sqrt{\phi(x_0)-\phi(x_{\bm{\tau}_i}^k)}
 +|\bm{\tau}_j| \sqrt{\phi(x_0)-\phi(x_{\bm{\tau}_j}^l)} \right).
\end{align*}
Thanks to the a priori estimate (Lemma~\ref{lm:L322}), we have
\[ \max\{ \phi(x_0)-\phi(x_{\bm{\tau}_i}^k),\phi(x_0)-\phi(x_{\bm{\tau}_j}^l) \}
 \le C=C(x_0,\tau_*(\phi),\phi(x_0),T). \]
Therefore we conclude that $\bm{\bar{d}_{\tau_i \tau_j}}(t,t)$ tends to $0$ as $i,j \to \infty$,
uniformly in $t \in [0,T]$.
$\qedd$
\end{proof}

By virtue of the uniqueness, we can define the \emph{gradient flow} operator
\begin{equation}\label{eq:Gtx}
\mathcal{G}:[0,\infty) \times D(\phi) \lra D(\phi)
\end{equation}
by $\mathcal{G}(t,x_0):=\xi(t)$, where $\xi:[0,\infty) \lra X$ is the unique gradient curve with $\xi(0)=x_0$
given by Theorem~\ref{th:T422}.
Then the semigroup property:
\[ \mathcal{G}\big( t,\mathcal{G}(s,x_0) \big) =\mathcal{G}(s+t,x_0)\quad \text{for all}\ s,t \ge 0  \]
also follows from the uniqueness of gradient curves.

One can immediately obtain the following (rough) error estimate from the proof of Theorem~\ref{th:T422}
(compare this with \cite[Theorem~4.0.9]{AGSbook}).

\begin{corollary}[An error estimate]\label{cr:C423}
Let $\lambda \le 0$ and $|\bm{\tau}|<\tau_*(\phi)$, fix $x_0 \in D(\phi)$,
and put $\xi(t):=\mathcal{G}(t,x_0)$.
Then we have
\[ \bm{\bar{d}_{\tau}}^2\big(t;\xi(t) \big) \le e^{-2\lambda t}
 \left( \sqrt{1-K'-\frac{2\lambda}{3} |\bm{\tau}|} +\sqrt{-\frac{4\lambda}{3} |\bm{\tau}|} \right)^2
 |\bm{\tau}| \left\{ \phi(x_0)-\phi\big( \bm{\bar{x}_{\tau}}(t) \big) \right\} \]
for all $t>0$, where $K':=\min\{0,K\}$.
\end{corollary}

\begin{proof}
Taking the limit of \eqref{eq:t/s-2} as $|\bm{\sigma}| \to 0$
and using the estimates in the proof of Theorem~\ref{th:T422},
we have for $t \in (t_{\bm{\tau}}^{k-1},t_{\bm{\tau}}^k]$
\begin{align*}
e^{\lambda t} \bm{\bar{d}_{\tau}}\big(t;\xi(t) \big)
&\le \left\{ \int_0^t
 (2\mathscr{R}_{\bm{\tau},K} -\lambda \mathscr{D}^2_{\bm{\tau}}) \,ds \right\}^{1/2}
 -2\lambda \int_0^t e^{\lambda s} \mathscr{D}_{\bm{\tau}}(s) \,ds \\
&\le \left\{ \left( 1-K'-\frac{2\lambda}{3}|\bm{\tau}| \right) |\bm{\tau}|
 \{\phi(x_0)-\phi(x_{\bm{\tau}}^k) \} \right\}^{1/2} \\
&\quad +\sqrt{-\frac{4\lambda}{3}} |\bm{\tau}| \sqrt{\phi(x_0)-\phi(x_{\bm{\tau}}^k)} \\
&= \left( \sqrt{1-K'-\frac{2\lambda}{3} |\bm{\tau}|} +\sqrt{-\frac{4\lambda}{3} |\bm{\tau}|} \right)
 \sqrt{|\bm{\tau}|} \sqrt{\phi(x_0)-\phi\big( \bm{\bar{x}_{\tau}}(t) \big)}.
\end{align*}
This completes the proof.
$\qedd$
\end{proof}

\subsection{Contraction property}\label{ssc:cont}

Coming back to the discrete scheme, we show the following lemma,
which readily implies the \emph{contraction property} of $\mathcal{G}$ (Theorem~\ref{th:P431}).
Set
\[ \lambda_{\bm{\tau}}:=\frac{\log(1+\lambda |\bm{\tau}|)}{|\bm{\tau}|} \]
assuming $\lambda|\bm{\tau}|>-1$ (if $\lambda<0$),
and observe that $\lambda_{\bm{\tau}} \le \lambda$ and
\begin{equation}\label{eq:lamtau}
\frac{\log(1+\lambda \tau_k)}{\tau_k} \ge \lambda_{\bm{\tau}}
\end{equation}
for all $k \in \N$ (see \cite[Lemma~3.4.1]{AGSbook}).

\begin{lemma}[Discrete contraction estimate]\label{lm:L424}
Take $x_0,y_0 \in D(\phi)$ and $\mathscr{P}_{\bm{\tau}}$
with $|\bm{\tau}|<\tau_*(\phi)/8$.
Assume $\lambda |\bm{\tau}|>-1$ if $\lambda<0$.
Then we have, for $t \in (t_{\bm{\tau}}^{k-1},t_{\bm{\tau}}^k]$ with $t_{\bm{\tau}}^k \le T$,
\begin{align*}
e^{2(\lambda_{\bm{\tau}}t+\lambda^-_{\bm{\tau}}|\bm{\tau}|)}
 d^2\big( \bm{\bar{x}_{\tau}}(t),\bm{\bar{y}_{\tau}}(t) \big)
&\le d^2(x_0,y_0) +2e^{2\lambda^+_{\bm{\tau}}t_{\bm{\tau}}^k} |\bm{\tau}|
 \{ \phi(y_0)-\phi(y_{\bm{\tau}}^k) \} \\
&\quad -K' e^{2\lambda^+_{\bm{\tau}}t_{\bm{\tau}}^k}
 |\bm{\tau}| \{ \phi(x_0)-\phi(x_{\bm{\tau}}^k)+\phi(y_0)-\phi(y_{\bm{\tau}}^k) \} \\
&\quad +O_{x_0,y_0,T}(\sqrt{|\bm{\tau}|}),
\end{align*}
where $K':=\min\{0,K\}$, $\lambda^-_{\bm{\tau}}:=\min\{0,\lambda_{\bm{\tau}}\}$
and $\lambda^+_{\bm{\tau}}:=\max\{0,\lambda_{\bm{\tau}}\}$.
\end{lemma}

\begin{proof}
The proof is along the line of \cite[Lemma~4.2.4]{AGSbook}.
Applying Lemma~\ref{lm:key}, we have for each $k \in \N$
\begin{align*}
d^2(x_{\bm{\tau}}^k,y_{\bm{\tau}}^{k-1})
&\le d^2(x_{\bm{\tau}}^{k-1},y_{\bm{\tau}}^{k-1}) -\lambda \tau_k d^2(x_{\bm{\tau}}^k,y_{\bm{\tau}}^{k-1})
 +2\tau_k \{ \phi(y_{\bm{\tau}}^{k-1})-\phi(x_{\bm{\tau}}^k) \} \\
&\quad -K' \tau_k \{ \phi(x_{\bm{\tau}}^{k-1})-\phi(x_{\bm{\tau}}^k) \},
\end{align*}
and
\begin{align*}
d^2(x_{\bm{\tau}}^k,y_{\bm{\tau}}^k)
&\le d^2(x_{\bm{\tau}}^k,y_{\bm{\tau}}^{k-1}) -\lambda \tau_k d^2(x_{\bm{\tau}}^k,y_{\bm{\tau}}^k)
 +2\tau_k \{ \phi(x_{\bm{\tau}}^k)-\phi(y_{\bm{\tau}}^k) \} \\
&\quad -K' \tau_k \{ \phi(y_{\bm{\tau}}^{k-1})-\phi(y_{\bm{\tau}}^k) \}.
\end{align*}
Thus we have
\begin{align*}
(1+\lambda \tau_k)d^2(x_{\bm{\tau}}^k,y_{\bm{\tau}}^k)
&\le d^2(x_{\bm{\tau}}^{k-1},y_{\bm{\tau}}^{k-1}) -\lambda \tau_k d^2(x_{\bm{\tau}}^k,y_{\bm{\tau}}^{k-1})
 +2\tau_k \{ \phi(y_{\bm{\tau}}^{k-1})-\phi(y_{\bm{\tau}}^k) \} \\
&\quad -K' \tau_k \{ \phi(x_{\bm{\tau}}^{k-1})-\phi(x_{\bm{\tau}}^k)
+\phi(y_{\bm{\tau}}^{k-1})-\phi(y_{\bm{\tau}}^k) \}.
\end{align*}
Note that
\begin{align*}
|d^2(x_{\bm{\tau}}^k,y_{\bm{\tau}}^{k-1})-d^2(x_{\bm{\tau}}^{k-1},y_{\bm{\tau}}^{k-1})|
&\le \{ d(x_{\bm{\tau}}^k,y_{\bm{\tau}}^{k-1})+d(x_{\bm{\tau}}^{k-1},y_{\bm{\tau}}^{k-1}) \}
 d(x_{\bm{\tau}}^k,x_{\bm{\tau}}^{k-1}) \\
&=O_{x_0,y_0,T}(\sqrt{|\bm{\tau}|})
\end{align*}
by the a priori estimate (Lemma~\ref{lm:L322}, \eqref{eq:apri}).
Together with $1-\lambda \tau_k \le (1+\lambda \tau_k)^{-1}$, this implies
\begin{align*}
(1+\lambda \tau_k)d^2(x_{\bm{\tau}}^k,y_{\bm{\tau}}^k)
&\le \frac{1}{1+\lambda \tau_k} d^2(x_{\bm{\tau}}^{k-1},y_{\bm{\tau}}^{k-1})
 +2\tau_k \{ \phi(y_{\bm{\tau}}^{k-1})-\phi(y_{\bm{\tau}}^k) \} \\
&\quad -K' \tau_k \{ \phi(x_{\bm{\tau}}^{k-1})-\phi(x_{\bm{\tau}}^k)
+\phi(y_{\bm{\tau}}^{k-1})-\phi(y_{\bm{\tau}}^k) \} \\
&\quad +\tau_k \cdot O_{x_0,y_0,T}(\sqrt{|\bm{\tau}|}).
\end{align*}
Multiplying both sides by
$e^{\lambda_{\bm{\tau}}(2t_{\bm{\tau}}^{k-1}+\tau_k)}
 =e^{\lambda_{\bm{\tau}}(2t_{\bm{\tau}}^k -\tau_k)}$
yields that, since
\[ (1+\lambda \tau_k) e^{-\lambda_{\bm{\tau}}\tau_k} e^{2\lambda_{\bm{\tau}}t_{\bm{\tau}}^k}
 \ge  e^{2\lambda_{\bm{\tau}}t_{\bm{\tau}}^k}, \qquad
 e^{2\lambda_{\bm{\tau}}t_{\bm{\tau}}^{k-1}}
 \frac{e^{\lambda_{\bm{\tau}}\tau_k}}{1+\lambda \tau_k}
 \le e^{2\lambda_{\bm{\tau}}t_{\bm{\tau}}^{k-1}} \]
by \eqref{eq:lamtau},
\begin{align*}
e^{2\lambda_{\bm{\tau}} t_{\bm{\tau}}^k} d^2(x_{\bm{\tau}}^k,y_{\bm{\tau}}^k)
&\le e^{2\lambda_{\bm{\tau}} t_{\bm{\tau}}^{k-1}} d^2(x_{\bm{\tau}}^{k-1},y_{\bm{\tau}}^{k-1})
 +2e^{\lambda_{\bm{\tau}}(t_{\bm{\tau}}^{k-1}+t_{\bm{\tau}}^k)} \tau_k
 \{ \phi(y_{\bm{\tau}}^{k-1})-\phi(y_{\bm{\tau}}^k) \} \\
&\quad -K' e^{\lambda_{\bm{\tau}}(t_{\bm{\tau}}^{k-1}+t_{\bm{\tau}}^k)}
 \tau_k \{ \phi(x_{\bm{\tau}}^{k-1})-\phi(x_{\bm{\tau}}^k)
+\phi(y_{\bm{\tau}}^{k-1})-\phi(y_{\bm{\tau}}^k) \} \\
&\quad +\tau_k \cdot O_{x_0,y_0,T}(\sqrt{|\bm{\tau}|}).
\end{align*}
Summing up, for $t \in (t_{\bm{\tau}}^{k-1},t_{\bm{\tau}}^k]$, we obtain the desired estimate
\begin{align*}
e^{2(\lambda_{\bm{\tau}}t+\lambda^-_{\bm{\tau}}|\bm{\tau}|)}
 d^2\big( \bm{\bar{x}_{\tau}}(t),\bm{\bar{y}_{\tau}}(t) \big)
&\le e^{2\lambda_{\bm{\tau}} t_{\bm{\tau}}^k} d^2(x_{\bm{\tau}}^k,y_{\bm{\tau}}^k) \\
&\le d^2(x_0,y_0) +2e^{2\lambda^+_{\bm{\tau}}t_{\bm{\tau}}^k} |\bm{\tau}|
 \{ \phi(y_0)-\phi(y_{\bm{\tau}}^k) \} \\
&\quad -K' e^{2\lambda^+_{\bm{\tau}}t_{\bm{\tau}}^k}
 |\bm{\tau}| \{ \phi(x_0)-\phi(x_{\bm{\tau}}^k)+\phi(y_0)-\phi(y_{\bm{\tau}}^k) \} \\
&\quad +t_{\bm{\tau}}^k \cdot O_{x_0,y_0,T}(\sqrt{|\bm{\tau}|}).
\end{align*}
$\qedd$
\end{proof}

\begin{theorem}[Contraction property]\label{th:P431}
Take $x_0,y_0 \in D(\phi)$ and put $\xi(t):=\mathcal{G}(t,x_0)$ and $\zeta(t):=\mathcal{G}(t,y_0)$.
Then we have, for any $t>0$,
\begin{equation}\label{eq:cont}
d\big( \xi(t),\zeta(t) \big) \le e^{-\lambda t} d(x_0,y_0).
\end{equation}
\end{theorem}

\begin{proof}
Take the limit as $|\bm{\tau}| \to 0$ in Lemma~\ref{lm:L424}.
Then the claim follows from $\lim_{|\bm{\tau}| \to 0}\lambda_{\bm{\tau}}=\lambda$
and the a priori estimate in Lemma~\ref{lm:L322} (which bounds $\phi(x_0)-\phi(x_{\bm{\tau}}^k)$).
$\qedd$
\end{proof}

The contraction property allows us to take the continuous limit
\[ \mathcal{G}:[0,\infty) \times \overline{D(\phi)} \lra \overline{D(\phi)} \]
of the gradient flow operator in \eqref{eq:Gtx},
which again enjoys the semigroup property as well as the contraction property \eqref{eq:cont}.
One can alternatively derive the contraction property
from the evolution variational inequality \eqref{eq:EVI} below,
whereas we think that this direct proof and the discrete estimate in Lemma~\ref{lm:L424}
are worthwhile as well.

\subsection{Evolution variational inequality}\label{ssc:EVI}

Similarly to \cite[Theorem~4.3.2]{AGSbook}, taking the limit of Theorem~\ref{th:T414},
we obtain the following.

\begin{theorem}[Evolution variational inequality]\label{th:T432}
Take $x_0 \in D(\phi)$ and put $\xi(t):=\mathcal{G}(t,x_0)$.
Then we have
\begin{equation}\label{eq:EVI}
\limsup_{\ve \downarrow 0}\frac{d^2(\xi(t+\ve),y)-d^2(\xi(t),y)}{2\ve}
 +\frac{\lambda}{2} d^2\big( \xi(t),y \big) +\phi\big( \xi(t) \big) \le \phi(y)
\end{equation}
for all $y \in D(\phi)$ and $t>0$.
In particular,
\[ \frac{1}{2} \frac{d}{dt} \big[ d^2\big( \xi(t),y \big) \big] +\frac{\lambda}{2} d^2\big( \xi(t),y \big)
 +\phi\big( \xi(t) \big) \le \phi(y) \]
for all $y \in D(\phi)$ and almost all $t>0$.
\end{theorem}

\begin{proof}
By recalling the estimate of the integral of $\mathscr{R}_{\bm{\tau},K}$ in the proof of Theorem~\ref{th:T422},
integration in $t \in [S,T]$ of Theorem~\ref{th:T414} gives
\begin{align*}
&\frac{1}{2}\bm{\bar{d}_{\tau}}^2(T;y) -\frac{1}{2}\bm{\bar{d}_{\tau}}^2(S;y)
 +\int_S^T \left\{ \frac{\lambda}{2}d^2\big( \bm{\bar{x}_{\tau}}(t),y \big)
 +\bm{\bar{\phi}_{\tau}}(t) \right\} dt \\
&\le (T-S)\phi(y)
 +\left( \frac{1}{2} -\frac{K'}{2} \right) |\bm{\tau}|
 \left\{ \phi\big( \bm{\bar{x}_{\tau}}(S-|\bm{\tau}|) \big)-\phi\big( \bm{\bar{x}_{\tau}}(T) \big) \right\}.
\end{align*}
Note that $\bm{\bar{\phi}_{\tau}}$ is uniformly bounded on $[0,T]$
thanks to the a priori estimate (Lemma~\ref{lm:L322}).
Thus we have
\[ \int_S^T \phi \circ \xi \,dt \le \int_S^T \liminf_{|\bm{\tau}| \to 0} \bm{\bar{\phi}_{\tau}} \,dt
 \le \liminf_{|\bm{\tau}| \to 0} \int_S^T \bm{\bar{\phi}_{\tau}} \,dt \]
by the lower semi-continuity of $\phi$ and Fatou's lemma.
Therefore letting $|\bm{\tau}| \downarrow 0$ shows the integrated form of the evolution variational inequality:
\[ \frac{d^2(\xi(T),y) -d^2(\xi(S),y)}{2}
 +\int_S^T \left\{ \frac{\lambda}{2}d^2\big( \xi(t),y \big) +\phi\big( \xi(t) \big) \right\} dt
 \le (T-S)\phi(y). \]
Dividing both sides by $T-S$ and letting $T-S \downarrow 0$, we obtain the desired inequality
by the lower semi-continuity of $\phi$.
(We remark that $\phi \circ \xi$ is in fact continuous; see \cite[Theorem~2.4.15]{AGSbook}.)
$\qedd$
\end{proof}

\subsection{Stationary points and large time behavior of the flow}\label{ssc:Mayer}

In this subsection, following the argumentation in \cite{Ma} (on $\CAT(0)$-spaces),
we study stationary points and the large time behavior of the gradient flow $\mathcal{G}$.
Since the fundamental properties of the flow,
for establishing those the $\CAT(0)$-property is used in \cite[Section~1]{Ma},
is already in hand, we can follow the line of \cite[Section~2]{Ma} almost verbatim.
We begin with a consequence of the evolution variational inequality (Theorem~\ref{th:T432})
corresponding to \cite[Lemma~2.8]{Ma}.

\begin{lemma}\label{lm:MaL28}
Take $x_0 \in D(\phi)$ and put $\xi(t):=\mathcal{G}(t,x_0)$.
Then we have
\[ d^2\big( \xi(T),y \big) \le e^{-\lambda T}d^2(x_0,y)
 +2e^{-\lambda T} \int_0^T e^{\lambda t} \big\{ \phi(y) -\phi\big( \xi(t) \big) \big\} \,dt \]
for all $T>0$ and $y \in D(\phi)$.
In particular, we have
\[ d^2\big( \xi(T),y \big) \le e^{-\lambda T}d^2(x_0,y)
 -\frac{2(1-e^{-\lambda T})}{\lambda} \big\{ \phi\big( \xi(T) \big)-\phi(y) \big\}, \]
where $(1-e^{-\lambda T})/\lambda:=T$ when $\lambda=0$.
\end{lemma}

\begin{proof}
Rewrite \eqref{eq:EVI} as
\[ \limsup_{\ve \downarrow 0}\frac{e^{\lambda(t+\ve)}d^2(\xi(t+\ve),y)-e^{\lambda t}d^2(\xi(t),y)}{2\ve}
 \le e^{\lambda t} \big\{ \phi(y) -\phi\big( \xi(t) \big) \big\}, \]
which implies the first claim.
The second claim readily follows from this and the fact that $\phi(\xi(t))$ is non-increasing in $t$.
$\qedd$
\end{proof}

By the above lemma, one can show a characterization of stationary points
of the flow $\mathcal{G}$ in terms of the local slope $|\nabla\phi|$.

\begin{theorem}[A characterization of stationary points]\label{th:MaT212}
A point $x_0 \in D(\phi)$ satisfies $|\nabla \phi|(x_0)=0$ if and only if
$\mathcal{G}(t,x_0)=x_0$ for all $t>0$.
\end{theorem}

\begin{proof}
The ``only if'' part is a consequence of \cite[Lemma~2.11]{Ma} asserting that
\begin{equation}\label{eq:MaL211}
|\nabla\phi|(x_0)=0 \quad \text{if and only if} \quad
 \sup_{x \neq x_0}\frac{\phi(x_0)-\phi(x)}{d^2(x_0,x)}<\infty,
\end{equation}
for which only the $\lambda$-convexity of $\phi$ is used.
The ``if'' part follows from the same relation \eqref{eq:MaL211} and Lemma~\ref{lm:MaL28},
noticing that $(1-e^{-\lambda T})/\lambda \ge T$ for $\lambda <0$.
See \cite[Theorem~2.12]{Ma} for details.
$\qedd$
\end{proof}

It is natural to expect that, if $\phi \circ \xi$ does not diverge to $-\infty$,
then $|\nabla\phi| \circ \xi$ tends to $0$.
This is indeed the case as follows.

\begin{lemma}\label{lm:MaL227}
Assume $\lambda \le 0$, take $x_0 \in D(\phi)$ and put $\xi(t):=\mathcal{G}(t,x_0)$.
Then we have
\begin{equation}\label{eq:MaL227}
|\nabla\phi| \big( \xi(T) \big) -|\nabla\phi| \big( \xi(S) \big)
 \le -\sqrt{2}\lambda \int_S^T |\nabla\phi| \circ \xi \,dt
\end{equation}
for all $0<S<T$.
\end{lemma}

\begin{proof}
For any $x \in D(\phi)$ and $x_{\tau} \in J^{\phi}_{\tau}(x)$,
it follows from the $\lambda$-convexity of $\phi$ that
\[ |\nabla\phi|(x_{\tau}) \le |\nabla\phi|(x) -\lambda d(x,x_{\tau}) \]
(see \cite[Lemma~2.23]{Ma}).
Substituting $d^2(x,x_{\tau}) \le 2\tau\{\phi(x)-\phi(x_{\tau})\}$ and iterating this estimate,
one finds
\[ |\nabla\phi|(x_{\bm{\tau}}^N) \le |\nabla\phi|(x_0)
 -\lambda \sum_{k=1}^N \sqrt{2\tau_k \{\phi(x_{\bm{\tau}}^{k-1})-\phi(x_{\bm{\tau}}^k)\}}. \]
Applying the Cauchy--Schwarz inequality:
\[ \sum_{k=1}^N \sqrt{\tau_k \{\phi(x_{\bm{\tau}}^{k-1})-\phi(x_{\bm{\tau}}^k)\}}
 \le \sqrt{\sum_{k=1}^N \tau_k} \sqrt{\sum_{k=1}^N \{\phi(x_{\bm{\tau}}^{k-1})-\phi(x_{\bm{\tau}}^k)\}} \]
and taking the limit as $|\bm{\tau}| \to 0$, we have
\begin{equation}\label{eq:MaT224}
|\nabla\phi| \big( \xi(T) \big) -|\nabla\phi|(x_0)
 \le -\sqrt{2}\lambda\sqrt{T}\sqrt{\phi(x_0)-\phi \big( \xi(T) \big)}
\end{equation}
for all $T>0$, since $\phi$ and $|\nabla\phi|$ are lower semi-continuous
(by \cite[Proposition~2.25]{Ma} or \cite[Corollary~2.4.10]{AGSbook}).
By replacing $x_0$ with $\xi(S)$,
the implication from \eqref{eq:MaT224} to \eqref{eq:MaL227} is the same as \cite[Lemma~2.27]{Ma}.
$\qedd$
\end{proof}

\begin{theorem}[Large time behavior]\label{th:MaT230}
Take $x_0 \in D(\phi)$, put $\xi(t):=\mathcal{G}(t,x_0)$ and assume
$\lim_{t \to \infty}\phi(\xi(t))>-\infty$.
Then we have $\lim_{t \to \infty}|\nabla\phi|(\xi(t))=0$.
\end{theorem}

\begin{proof}
The proof is done by contradiction with the help of the estimate \eqref{eq:MaL227}
and the right continuity of $|\nabla\phi| \circ \xi$
(see \cite[Corollary~2.28]{Ma} or \cite[Theorem~2.4.15]{AGSbook}).
We refer to \cite[Theorem~2.30]{Ma} for details.
$\qedd$
\end{proof}

The following corollary is immediate, see \cite[Corollary~2.31]{Ma}.

\begin{corollary}\label{cr:MaC231}
Take $x_0 \in D(\phi)$, put $\xi(t):=\mathcal{G}(t,x_0)$ and assume that
there is a sequence $\{t_n\}_{n \in \N}$ such that $\lim_{n \to \infty}t_n=\infty$
and $\{\xi(t_n)\}_{n \in \N}$ converges to a point $\bar{x}$.
Then $\bar{x}$ is a stationary point of $\phi$ $($in the sense of Theorem~$\ref{th:MaT212})$
and $\lim_{t \to \infty}\phi(\xi(t))=\phi(\bar{x})$.
\end{corollary}

In general, $\lim_{t \to \infty}|\nabla\phi|(\xi(t))=0$ does not imply
the convergence to a stationary point.
One needs some compactness condition to find a stationary point,
see for instance \cite[Theorem~2.32]{Ma}.

\subsection{The case of $\CAT(1)$-spaces with diameter $\ge \pi/2$}\label{ssc:pi}

All the results in this section are generalized to complete $\CAT(1)$-spaces $(X,d)$ with $\diam X \ge \pi/2$.
First of all, given $x_0 \in D(\phi)$,
one can restrict the construction of the gradient curve in, say, the open ball $B(x_0,\pi/6)$.
Since $B(x_0,\pi/6)$ is (geodesically) convex,
the squared distance function in this ball is $K$-convex with $K=K(\pi/3)>0$ from Lemma~\ref{lm:Kconv},
and we obtain the gradient curve $\xi$ with $\xi(0)=x_0$.
Once $\xi(t)$ hits the boundary $\del B(x_0,\pi/6)$ at $t=t_1$,
we restart the construction in $B(\xi(t_1),\pi/6)$.
We remark that $t_1 \ge (\pi/6)^2/(2C)$ by \eqref{eq:apri}.
Iterating this procedure gives the gradient curve $\xi:[0,\infty) \lra X$.

The contraction property and the evolution variational inequality
are globalized in a standard way as follows.
(Then Theorems~\ref{th:MaT212}, \ref{th:MaT230} also hold true since they are
based only on the evolution variational inequality.)

For the contraction property (Theorem~\ref{th:P431}),
if $d(x_0,y_0) \ge \pi/2$, then we consider a minimal geodesic $\gamma$ from $x_0$ to $y_0$
and choose points $z_0=x_0,z_1,\ldots,z_{m-1},z_m=y_0$ on $\gamma$
such that $\max_{1 \le l \le m} d(z_{l-1},z_l) \ll \pi e^{-|\lambda|T}/2$ for given $T>0$.
Applying Theorem~\ref{th:P431} to adjacent gradient curves $\xi_l:=\mathcal{G}(\cdot,z_l)$
shows $d(\xi_{l-1}(t),\xi_l(t)) \le e^{-\lambda t}d(z_{l-1},z_l)$ for $t \in [0,T]$.
This yields $d(\xi(t),\zeta(t)) \le e^{-\lambda t}d(x_0,y_0)$ by the triangle inequality.

For the evolution variational inequality (Theorem~\ref{th:T432}),
given a minimal geodesic $\gamma:[0,1] \lra X$ from $\xi(t)$ to $y$,
it is easy to see that \eqref{eq:EVI} for $y=\gamma(s)$ with small $s>0$
(so that $d(\xi(t),\gamma(s)) \ll \pi/2$) implies \eqref{eq:EVI} itself.
Indeed, since
\begin{align*}
&\frac{d^2(\xi(t+\ve),y)-d^2(\xi(t),y)}{2\ve} \\
&\le \frac{\{d(\xi(t+\ve),\gamma(s))+d(\gamma(s),y)\}^2 -\{d(\xi(t),\gamma(s))+d(\gamma(s),y)\}^2}{2\ve} \\
&= \frac{d^2(\xi(t+\ve),\gamma(s)) -d^2(\xi(t),\gamma(s))}{2\ve}
 +\frac{d(\xi(t+\ve),\gamma(s)) -d(\xi(t),\gamma(s))}{\ve} d\big( \gamma(s),y \big) \\
&= \frac{d^2(\xi(t+\ve),\gamma(s)) -d^2(\xi(t),\gamma(s))}{2\ve}
 \left( 1+\frac{2d(\gamma(s),y)}{d(\xi(t+\ve),\gamma(s))+d(\xi(t),\gamma(s))} \right),
\end{align*}
we obtain from \eqref{eq:EVI} with $y=\gamma(s)$ that
\begin{align*}
&\limsup_{\ve \downarrow 0} \frac{d^2(\xi(t+\ve),y)-d^2(\xi(t),y)}{2\ve} \\
&\le \left\{ -\frac{\lambda}{2}d^2\big( \xi(t),\gamma(s) \big) -\phi\big( \xi(t) \big) +\phi\big( \gamma(s) \big) \right\}
 \left( 1+\frac{d(\gamma(s),y)}{d(\xi(t),\gamma(s))} \right) \\
&\le \frac{1}{s}\left\{ -\frac{\lambda}{2}s^2 d^2\big( \xi(t),y \big) -\phi\big( \xi(t) \big)
 +(1-s)\phi\big( \xi(t) \big) +s\phi(y) -\frac{\lambda}{2}(1-s)sd^2\big( \xi(t),y \big) \right\} \\
&= -\frac{\lambda}{2} d^2\big( \xi(t),y \big) -\phi\big( \xi(t) \big) +\phi(y).
\end{align*}

\section{A Trotter--Kato product formula}\label{sc:TK}

This final section is devoted to a further application of our key lemma:
a Trotter--Kato product formula for semi-convex functions.
See \cite{KM} for the classical setting of convex functions on Hilbert spaces.
The Trotter--Kato product formula on metric spaces was established
by Stojkovic~\cite{St} for convex functions on $\CAT(0)$-spaces  in terms of ultra-limits
(see also a recent result \cite{Ba} in terms of weak convergence),
and by Cl\'ement and Maas~\cite{CM} for functions satisfying
the assertion of our key lemma (Lemma~\ref{lm:key}) with $K=2$ and $\lambda=0$
(thus including convex functions on $\CAT(0)$-spaces).
We stress that, similarly to the previous section,
both the squared distance function and potential functions are allowed to be semi-convex in our argument.

\subsection{Setting and the main theorem}\label{ssc:TK-thm}

\begin{assumption}\label{as:TK0}
Let $(X,d)$ be a complete metric space in either Case~\ref{case:CAT} or Case~\ref{case:Kcon}
(see the beginning of Section~\ref{sc:gf}),
and assume additionally $D:=\diam X<\infty$.
For $i=1,2$, we consider a lower semi-continuous, $\lambda_i$-convex function
$\phi_i:X \lra (-\infty,\infty]$ ($\lambda_i \in \R$) satisfying
$D(\phi_1) \cap D(\phi_2) \neq \emptyset$ and the compactness (Assumption~\ref{as:phi}\eqref{it:cpt}).
\end{assumption}

We remark that $\lambda_i$ can be negative.
The sum $\phi:=\phi_1+\phi_2$ is clearly lower semi-continuous, $(\lambda_1+\lambda_2)$-convex
and enjoys Assumption~\ref{as:phi}\eqref{it:cpt} (with $\tau_*(\phi)=\infty$)
since $\phi_i$ is bounded below (Remark~\ref{rm:bdd}) and
\[ \{x\in X \,|\, \phi(x) \le Q\} \subset \{x \in X \,|\, \phi_1(x) \le Q-\inf_X \phi_2\}. \]

Given $z_0 \in D(\phi)=D(\phi_1) \cap D(\phi_2)$ and a partition $\mathscr{P}_{\bm{\tau}}$,
we consider the discrete variational schemes for $\phi_1$ and $\phi_2$ in turn, namely
\begin{equation}\label{eq:TK-dgf}
z_{\bm{\tau}}^0:=z_0,\ \text{choose arbitrary}\
 \hat{z}_{\bm{\tau}}^k \in J^{\phi_1}_{\tau_k}(z_{\bm{\tau}}^{k-1})\ \text{and then}\
 z_{\bm{\tau}}^k \in J^{\phi_2}_{\tau_k}(\hat{z}_{\bm{\tau}}^k)\ \text{for}\ k \in \N.
\end{equation}
If $\phi_1=\phi_2$, then this scheme reduces to \eqref{eq:dgf} for $\phi$ with respect to the partition:
\[ \left\{0 < \frac{t_{\bm{\tau}}^1}{2} <t_{\bm{\tau}}^1 <\frac{t_{\bm{\tau}}^1 +t_{\bm{\tau}}^2}{2}
 <t_{\bm{\tau}}^2 <\cdots \right\}. \]
The \emph{Trotter--Kato product formula} asserts that $\{z_{\bm{\tau}}^k\}_{k \ge 0}$
converges to the gradient curve of $\phi$ emanating from $z_0$ in an appropriate sense.
This is useful when $\phi_1$ and $\phi_2$ are easier to handle than their sum $\phi$.
An additional difficulty (in the discrete scheme) compared with the direct variational approximation for $\phi$
is that we have a priori no control of $\phi_2(\hat{z}_{\bm{\tau}}^k)-\phi_2(z_{\bm{\tau}}^{k-1})$
and $\phi_1(z_{\bm{\tau}}^k)-\phi_1(\hat{z}_{\bm{\tau}}^k)$
(both are being nonpositive if $\phi_1=\phi_2$).
Thus we suppose the following condition:

\begin{assumption}\label{as:TK}
Given $z_0 \in D(\phi)$ and a partition $\mathscr{P}_{\bm{\tau}}$, set
\[ \delta_{\bm{\tau}}^k(z_0) :=\max\{0,\phi_2(\hat{z}_{\bm{\tau}}^k)-\phi_2(z_{\bm{\tau}}^{k-1}),
 \phi_1(z_{\bm{\tau}}^k)-\phi_1(\hat{z}_{\bm{\tau}}^k)\} \]
for $k \in \N$ by suppressing the dependence on the choice of
$\{ \hat{z}_{\bm{\tau}}^k,z_{\bm{\tau}}^k\}_{k \in \N}$.
Assume that, for any $\ve,T>0$, there is $\Delta_{\ve}^T(z_0) < \infty$ such that
\[ \sum_{k=1}^N \delta_{\bm{\tau}}^k(z_0) \le \Delta_{\ve}^T(z_0) \]
for any $\mathscr{P}_{\bm{\tau}}$ with $|\bm{\tau}|<\ve$, $N \in \N$ with $t_{\bm{\tau}}^N \le T$,
and for any solution $\{ \hat{z}_{\bm{\tau}}^k,z_{\bm{\tau}}^k\}_{k \in \N}$ to \eqref{eq:TK-dgf}.
This in particular guarantees that $\hat{z}_{\bm{\tau}}^k \in D(\phi)$
and $z_{\bm{\tau}}^k \in D(\phi)$.
\end{assumption}

\begin{example}\label{ex:TK}
One of the simplest examples satisfying Assumption~\ref{as:TK} is pairs of Lipschitz functions.
If both $\phi_1$ and $\phi_2$ are $L$-Lipschitz, then
\[ d^2(z_{\bm{\tau}}^{k-1},\hat{z}_{\bm{\tau}}^k)
 \le 2\tau_k\{\phi_1(z_{\bm{\tau}}^{k-1}) -\phi_1(\hat{z}_{\bm{\tau}}^k)\}
 \le 2\tau_k Ld(z_{\bm{\tau}}^{k-1},\hat{z}_{\bm{\tau}}^k). \]
Hence $d(z_{\bm{\tau}}^{k-1},\hat{z}_{\bm{\tau}}^k) \le 2L\tau_k$
and similarly $d(\hat{z}_{\bm{\tau}}^k,z_{\bm{\tau}}^k) \le 2L\tau_k$.
Thus we find
\[ \sum_{k=1}^N \delta_{\bm{\tau}}^k(z_0)
 \le \sum_{k=1}^N L \max\{d(z_{\bm{\tau}}^{k-1},\hat{z}_{\bm{\tau}}^k),
 d(\hat{z}_{\bm{\tau}}^k,z_{\bm{\tau}}^k)\}
 \le 2L^2t_{\bm{\tau}}^N. \]
Notice that $\Delta_{\ve}^T(z_0)$ is taken independently from $\ve$ and $z_0$ in this case.
See \cite[Proposition~1.7]{CM} for other examples.
\end{example}

Our assumptions are comparable with those in \cite{CM}.
(For the sake of simplicity,
we do not intend to minimize the assumptions in this section.)
To state the main theorem of the section, we introduce the interpolated curve $\bm{\bar{z}_{\tau}}$
similarly to \S \ref{ssc:piece}:
\[ \bm{\bar{z}_{\tau}}(0):=z_0, \qquad
 \bm{\bar{z}_{\tau}}(t):=z_{\bm{\tau}}^k\quad \text{for}\ t \in (t_{\bm{\tau}}^{k-1},t_{\bm{\tau}}^k]. \]

\begin{theorem}[A Trotter--Kato product formula]\label{th:TK}
Let Assumptions~$\ref{as:TK0}$, $\ref{as:TK}$ be satisfied.
Given $z_0 \in D(\phi)$, the curve
$\bm{\bar{z}_{\tau}}$ converges to the gradient curve $\xi:=\mathcal{G}(\cdot,z_0)$ of $\phi$
$($constructed in the previous section$)$
as $|\bm{\tau}| \to 0$ uniformly on each bounded interval $[0,T]$.
\end{theorem}

Similarly to the previous section, we will discuss under the global $K$-convexity
of the squared distance function.
Thus $\diam X<\pi$ is implicitly assumed in Case~\ref{case:CAT},
however, this costs no generality as we explained in \S \ref{ssc:pi}.

\subsection{Preliminary estimates}\label{ssc:TK-pre}

Fix $z_0 \in D(\phi)$, $\mathscr{P}_{\bm{\tau}}$ with $|\bm{\tau}|<\ve$
and $\{ \hat{z}_{\bm{\tau}}^k,z_{\bm{\tau}}^k\}_{k \in \N}$ solving \eqref{eq:TK-dgf}.

\begin{lemma}\label{lm:bound}
\begin{enumerate}[{\rm (i)}]
\item
For each $N \in \N$ with $t_{\bm{\tau}}^N \le T$, we have
\[ \max_{i=1,2}\{ \phi_i(z_{\bm{\tau}}^N)-\phi_i(z_0),\phi_i(\hat{z}_{\bm{\tau}}^N)-\phi_i(z_0) \}
 \le \sum_{k=1}^N \delta_{\bm{\tau}}^k(z_0) \le \Delta^T_{\ve}(z_0). \]
\item
For any $l \le k$ with $t_{\bm{\tau}}^k \le T$, we have
\begin{align*}
&\max\{ d(z_{\bm{\tau}}^{l-1},\hat{z}_{\bm{\tau}}^k), d(z_{\bm{\tau}}^{l-1},z_{\bm{\tau}}^k),
 d(\hat{z}_{\bm{\tau}}^l,\hat{z}_{\bm{\tau}}^k), d(\hat{z}_{\bm{\tau}}^l,z_{\bm{\tau}}^k) \} \\
&\le \sqrt{2(t_{\bm{\tau}}^k-t_{\bm{\tau}}^{l-1})}
 \left\{ \sqrt{\phi_1(z_0) -\inf_X \phi_1 +2\Delta_{\ve}^T(z_0)}
 +\sqrt{\phi_2(z_0) -\inf_X \phi_2 +2\Delta_{\ve}^T(z_0)} \right\}.
\end{align*}
\end{enumerate}
\end{lemma}

\begin{proof}
(i)
This is straightforward from the definition of $\delta_{\bm{\tau}}^k(z_0)$.
We know $\phi_1(\hat{z}_{\bm{\tau}}^k) \le \phi_1(z_{\bm{\tau}}^{k-1})$ and hence
\[ \phi_1(z_{\bm{\tau}}^N) =\phi_1(z_0)
 +\sum_{k=1}^N \{\phi_1(z_{\bm{\tau}}^k)-\phi_1(\hat{z}_{\bm{\tau}}^k)
 +\phi_1(\hat{z}_{\bm{\tau}}^k)-\phi_1(z_{\bm{\tau}}^{k-1})\}
 \le \phi_1(z_0)+\sum_{k=1}^N \delta_{\bm{\tau}}^k(z_0). \]
Similarly we find
$\phi_1(\hat{z}_{\bm{\tau}}^N) \le \phi_1(z_0)+\sum_{k=1}^{N-1} \delta_{\bm{\tau}}^k(z_0)$
and
\[ \phi_2(z_{\bm{\tau}}^N) \le
 \phi_2(\hat{z}_{\bm{\tau}}^N) \le \phi_2(z_0)+\sum_{k=1}^N \delta_{\bm{\tau}}^k(z_0). \]

(ii)
It follows from the Cauchy--Schwarz inequality that
\begin{align*}
&\max\{ d(z_{\bm{\tau}}^{l-1},\hat{z}_{\bm{\tau}}^k), d(z_{\bm{\tau}}^{l-1},z_{\bm{\tau}}^k),
 d(\hat{z}_{\bm{\tau}}^l,\hat{z}_{\bm{\tau}}^k), d(\hat{z}_{\bm{\tau}}^l,z_{\bm{\tau}}^k) \}
 \le \sum_{m=l}^k \{d(z_{\bm{\tau}}^{m-1},\hat{z}_{\bm{\tau}}^m)
 +d(\hat{z}_{\bm{\tau}}^m,z_{\bm{\tau}}^m)\} \\
&\le \sqrt{\sum_{m=l}^k 2\tau_m} \left\{ \sqrt{ \sum_{m=l}^k
 \frac{d^2(z_{\bm{\tau}}^{m-1},\hat{z}_{\bm{\tau}}^m)}{2\tau_m}}
 +\sqrt{ \sum_{m=l}^k \frac{d^2(\hat{z}_{\bm{\tau}}^m,z_{\bm{\tau}}^m)}{2\tau_m}}
 \right\} \\
&\le \sqrt{2(t_{\bm{\tau}}^k-t_{\bm{\tau}}^{l-1})}
 \left\{ \sqrt{\sum_{m=l}^k \{\phi_1(z_{\bm{\tau}}^{m-1})-\phi_1(\hat{z}_{\bm{\tau}}^m)\}}
 +\sqrt{\sum_{m=l}^k \{\phi_2(\hat{z}_{\bm{\tau}}^m)-\phi_2(z_{\bm{\tau}}^m)\}} \right\}.
\end{align*}
Note that, by (i),
\begin{align*}
\sum_{m=l}^k \{\phi_1(z_{\bm{\tau}}^{m-1})-\phi_1(\hat{z}_{\bm{\tau}}^m)\}
&\le \sum_{m=l}^k \{\phi_1(z_{\bm{\tau}}^{m-1})-\phi_1(z_{\bm{\tau}}^m)
 +\delta_{\bm{\tau}}^m(z_0)\} \\
&= \phi_1(z_{\bm{\tau}}^{l-1})-\phi_1(z_{\bm{\tau}}^k) +\sum_{m=l}^k \delta_{\bm{\tau}}^m(z_0) \\
&\le \phi_1(z_0) -\inf_X \phi_1 +2\Delta_{\ve}^T(z_0).
\end{align*}
Similarly we obtain
$\sum_{m=l}^k \{\phi_2(\hat{z}_{\bm{\tau}}^m)-\phi_2(z_{\bm{\tau}}^m)\}
 \le \phi_2(z_0) -\inf_X \phi_2 +2\Delta_{\ve}^T(z_0)$.
This completes the proof.
$\qedd$
\end{proof}

Thanks to (ii) above, as $|\bm{\tau}| \to 0$,
we obtain the uniform convergence of
a subsequence of $\bm{\bar{z}_{\tau}}:[0,T] \lra X$ to a H\"older continuous curve
$\zeta:[0,T] \lra X$ with $\zeta(0)=z_0$.
Our goal is to show that $\zeta$ coincides with the gradient curve $\xi$ of $\phi$.
Then $\bm{\bar{z}_{\tau}}$ uniformly converges to $\xi$ as $|\bm{\tau}| \to 0$
without passing to subsequences.
Observe that the uniformity can be seen by contradiction;
the existence of $\rho>0$ such that
$\sup_{t \in [0,T]}d(\bm{\bar{z}}_{\bm{\tau}_i}(t),\xi(t)) \ge \rho$
for all $i$ with $\lim_{i \to 0}|\bm{\tau}_i| \to 0$ contradicts the uniform convergence of a subsequence
of $\{\bm{\bar{z}}_{\bm{\tau}_i}\}_{i \in \N}$.

The following key estimate can be thought of as a discrete version of the evolution variational inequality
(compare this with \cite[Lemma~2.1]{CM}).

\begin{lemma}\label{lm:TK}
Assume $\lambda_i |\bm{\tau}|>-1$ for $i=1,2$.
For any $w \in D(\phi)$ and $k \in \N$ with $t_{\bm{\tau}}^k \le T$, we have
\begin{align*}
e^{(\lambda_1^{\bm{\tau}} +\lambda_2^{\bm{\tau}})t_{\bm{\tau}}^k} d^2(z_{\bm{\tau}}^k,w)
&\le e^{(\lambda_1^{\bm{\tau}}
 +\lambda_2^{\bm{\tau}})t_{\bm{\tau}}^{k-1}} d^2(z_{\bm{\tau}}^{k-1},w)
 +2e^{\lambda_2^{\bm{\tau}} t_{\bm{\tau}}^{k-1}+\lambda_1^{\bm{\tau}} t_{\bm{\tau}}^k} \tau_k
 \{\phi(w)-\phi(z_{\bm{\tau}}^k)+\delta_{\bm{\tau}}^k(z_0)\} \\
&\quad -K' 
 e^{\lambda_2^{\bm{\tau}} t_{\bm{\tau}}^{k-1}+\lambda_1^{\bm{\tau}} t_{\bm{\tau}}^k}
 \tau_k \{\phi(z_{\bm{\tau}}^{k-1})-\phi(z_{\bm{\tau}}^k)+2\delta_{\bm{\tau}}^k(z_0)\}
 +\tau_k \cdot O_{z_0,\ve,T}(\sqrt{\tau_k}),
\end{align*}
where $K':=\min\{0,K\}$ and
\[ \lambda_i^{\bm{\tau}}:=\frac{\log(1+\lambda_i |\bm{\tau}|)}{|\bm{\tau}|},
 \quad i=1,2. \]
\end{lemma}

\begin{proof}
The proof is based on calculations similar to Lemma~\ref{lm:L424}.
Applying Lemma~\ref{lm:key} to the steps $z_{\bm{\tau}}^{k-1} \to \hat{z}_{\bm{\tau}}^k$
and $\hat{z}_{\bm{\tau}}^k \to z_{\bm{\tau}}^k$, we have
\begin{align*}
(1+\lambda_1 \tau_k) d^2(\hat{z}_{\bm{\tau}}^k,w) -d^2(z_{\bm{\tau}}^{k-1},w)
&\le 2\tau_k \{\phi_1(w)-\phi_1(\hat{z}_{\bm{\tau}}^k)\}
 -K' \tau_k \{\phi_1(z_{\bm{\tau}}^{k-1}) -\phi_1(\hat{z}_{\bm{\tau}}^k)\}, \\
(1+\lambda_2 \tau_k) d^2(z_{\bm{\tau}}^k,w) -d^2(\hat{z}_{\bm{\tau}}^k,w)
&\le 2\tau_k \{\phi_2(w)-\phi_2(z_{\bm{\tau}}^k)\}
 -K' \tau_k \{\phi_2(\hat{z}_{\bm{\tau}}^k) -\phi_2(z_{\bm{\tau}}^k)\}.
\end{align*}
Thus we find
\begin{align*}
(1+\lambda_2 \tau_k) d^2(z_{\bm{\tau}}^k,w)
&\le d^2(z_{\bm{\tau}}^{k-1},w) -\lambda_1 \tau_k d^2(\hat{z}_{\bm{\tau}}^k,w)
 +2\tau_k \{\phi(w) -\phi_1(\hat{z}_{\bm{\tau}}^k) -\phi_2(z_{\bm{\tau}}^k)\} \\
&\quad -K' \tau_k \{\phi_1(z_{\bm{\tau}}^{k-1}) -\phi_1(\hat{z}_{\bm{\tau}}^k)
 +\phi_2(\hat{z}_{\bm{\tau}}^k) -\phi_2(z_{\bm{\tau}}^k)\}.
\end{align*}
Note that, by Lemma~\ref{lm:bound}(i),
\begin{align*}
|d^2(\hat{z}_{\bm{\tau}}^k,w) -d^2(z_{\bm{\tau}}^{k-1},w)|
&\le \{d(\hat{z}_{\bm{\tau}}^k,w)+d(z_{\bm{\tau}}^{k-1},w)\}
 d(\hat{z}_{\bm{\tau}}^k,z_{\bm{\tau}}^{k-1}) \\
&\le 2D\sqrt{2\tau_k \{\phi_1(z_{\bm{\tau}}^{k-1}) -\phi_1(\hat{z}_{\bm{\tau}}^k)\}} \\
&\le 2D\sqrt{2\tau_k} \sqrt{\phi_1(z_0) +\Delta_{\ve}^T(z_0) -\inf_X \phi_1} \\
&= O_{z_0,\ve,T}(\sqrt{\tau_k}).
\end{align*}
Moreover,
\[ \phi(w) -\phi_1(\hat{z}_{\bm{\tau}}^k) -\phi_2(z_{\bm{\tau}}^k)
 \le \phi(w) -\phi(z_{\bm{\tau}}^k) +\delta_{\bm{\tau}}^k(z_0) \]
and similarly
\[ \phi_1(z_{\bm{\tau}}^{k-1}) -\phi_1(\hat{z}_{\bm{\tau}}^k)
 +\phi_2(\hat{z}_{\bm{\tau}}^k) -\phi_2(z_{\bm{\tau}}^k)
 \le \phi(z_{\bm{\tau}}^{k-1}) -\phi(z_{\bm{\tau}}^k) +2\delta_{\bm{\tau}}^k(z_0). \]
Combining these yields
\begin{align*}
(1+\lambda_2 \tau_k) d^2(z_{\bm{\tau}}^k,w)
&\le \frac{1}{1+\lambda_1 \tau_k} d^2(z_{\bm{\tau}}^{k-1},w)
 +2\tau_k \{\phi(w) -\phi(z_{\bm{\tau}}^k) +\delta_{\bm{\tau}}^k(z_0)\} \\
&\quad -K' \tau_k \{\phi(z_{\bm{\tau}}^{k-1}) -\phi(z_{\bm{\tau}}^k)
 +2\delta_{\bm{\tau}}^k(z_0)\} +\tau_k \cdot O_{z_0,\ve,T}(\sqrt{\tau_k}).
\end{align*}
Multiply both sides by
$e^{\lambda_2^{\bm{\tau}} t_{\bm{\tau}}^{k-1}+\lambda_1^{\bm{\tau}} t_{\bm{\tau}}^k}
 =e^{(\lambda_1^{\bm{\tau}}+\lambda_2^{\bm{\tau}})t_{\bm{\tau}}^{k-1}
 +\lambda_1^{\bm{\tau}} \tau_k}
 =e^{(\lambda_1^{\bm{\tau}}+\lambda_2^{\bm{\tau}})t_{\bm{\tau}}^k
 -\lambda_2^{\bm{\tau}} \tau_k}$.
Then, recalling \eqref{eq:lamtau}, we obtain the desired estimate.
$\qedd$
\end{proof}

\subsection{Proof of Theorem~\ref{th:TK}}\label{ssc:TK-prf}

By Lemma~\ref{lm:TK},
\begin{align*}
&\frac{d^2(z_{\bm{\tau}}^k,w)-d^2(z_{\bm{\tau}}^{k-1},w)}{2\tau_k} \\
&= \frac{e^{(\lambda_1^{\bm{\tau}}+\lambda_2^{\bm{\tau}})t_{\bm{\tau}}^k}
 d^2(z_{\bm{\tau}}^k,w)
 -e^{(\lambda_1^{\bm{\tau}}+\lambda_2^{\bm{\tau}})t_{\bm{\tau}}^{k-1}}
 d^2(z_{\bm{\tau}}^{k-1},w)}
 {2e^{(\lambda_1^{\bm{\tau}}+\lambda_2^{\bm{\tau}})t_{\bm{\tau}}^k}\tau_k}
 +\frac{e^{-(\lambda_1^{\bm{\tau}}+\lambda_2^{\bm{\tau}})\tau_k}-1}{2\tau_k}
 d^2(z_{\bm{\tau}}^{k-1},w) \\
&\le e^{-\lambda_2^{\bm{\tau}}\tau_k}
 \{\phi(w)-\phi(z_{\bm{\tau}}^k)+\delta_{\bm{\tau}}^k(z_0)\}
 -\frac{K'}{2} e^{-\lambda_2^{\bm{\tau}}\tau_k}
 \{\phi(z_{\bm{\tau}}^{k-1})-\phi(z_{\bm{\tau}}^k)+2\delta_{\bm{\tau}}^k(z_0)\} \\
&\quad +\frac{e^{-(\lambda_1^{\bm{\tau}}+\lambda_2^{\bm{\tau}})\tau_k}-1}{2\tau_k}
 d^2(z_{\bm{\tau}}^{k-1},w) +O(\sqrt{|\bm{\tau}|}) \\
&= \phi(w)-\phi(z_{\bm{\tau}}^k) -\frac{K'}{2}\{\phi(z_{\bm{\tau}}^{k-1})-\phi(z_{\bm{\tau}}^k)\}
 -\frac{\lambda_1+\lambda_2}{2} d^2(z_{\bm{\tau}}^{k-1},w)
 +(1-K')\delta_{\bm{\tau}}^k(z_0) \\
&\quad +O(\sqrt{|\bm{\tau}|}).
\end{align*}
We used the bound of $\phi(z_{\bm{\tau}}^k)$ (Lemma~\ref{lm:bound}(i))
to estimate the error terms.
Denote by $\{x_{\bm{\tau}}^k\}_{k \ge 0}$
a discrete solution of the variational scheme \eqref{eq:dgf} for $\phi$ with $x_{\bm{\tau}}^0=z_0$.
We recall from the proof of Theorem~\ref{th:T414} that, putting $\lambda:=\lambda_1+\lambda_2$,
\[ \frac{d^2(x_{\bm{\tau}}^k,y) -d^2(x_{\bm{\tau}}^{k-1},y)}{2\tau_k}
 \le \phi(y)-\phi(x_{\bm{\tau}}^k) -\frac{K'}{2}\{\phi(x_{\bm{\tau}}^{k-1})-\phi(x_{\bm{\tau}}^k)\}
 -\frac{\lambda}{2}d^2(x_{\bm{\tau}}^k,y). \]
Applying these inequalities with $w=x_{\bm{\tau}}^{k-1}$ and $y=z_{\bm{\tau}}^k$ to
\[ d^2(x_{\bm{\tau}}^N,z_{\bm{\tau}}^N)
 = \sum_{k=1}^N \left\{ d^2(x_{\bm{\tau}}^k,z_{\bm{\tau}}^k) -d^2(x_{\bm{\tau}}^{k-1},z_{\bm{\tau}}^k)
 +d^2(x_{\bm{\tau}}^{k-1},z_{\bm{\tau}}^k) -d^2(x_{\bm{\tau}}^{k-1},z_{\bm{\tau}}^{k-1}) \right\}, \]
we obtain for $N$ with $t_{\bm{\tau}}^N \le T$
\begin{align*}
d^2(x_{\bm{\tau}}^N,z_{\bm{\tau}}^N)
&\le \sum_{k=1}^N 2\tau_k \left\{
\phi(z_{\bm{\tau}}^k)-\phi(x_{\bm{\tau}}^k) -\frac{K'}{2}\{\phi(x_{\bm{\tau}}^{k-1})-\phi(x_{\bm{\tau}}^k)\}
 -\frac{\lambda}{2}d^2(x_{\bm{\tau}}^k,z_{\bm{\tau}}^k) \right\} \\
&\quad +\sum_{k=1}^N 2\tau_k \left\{
 \phi(x_{\bm{\tau}}^{k-1})-\phi(z_{\bm{\tau}}^k)
 -\frac{K'}{2}\{\phi(z_{\bm{\tau}}^{k-1})-\phi(z_{\bm{\tau}}^k)\}
 -\frac{\lambda}{2} d^2(x_{\bm{\tau}}^{k-1},z_{\bm{\tau}}^{k-1}) \right\} \\
&\quad +\sum_{k=1}^N 2\tau_k (1-K')\delta_{\bm{\tau}}^k(z_0)
 +t_{\bm{\tau}}^N \cdot O(\sqrt{|\bm{\tau}|}) \\
&\le (2-K') \sum_{k=1}^N \tau_k \{\phi(x_{\bm{\tau}}^{k-1})-\phi(x_{\bm{\tau}}^k)\}
 -K'\sum_{k=1}^N \tau_k \{\phi(z_{\bm{\tau}}^{k-1})-\phi(z_{\bm{\tau}}^k)\} \\
&\quad -\lambda \sum_{k=1}^N \tau_k
 \{d^2(x_{\bm{\tau}}^{k-1},z_{\bm{\tau}}^{k-1})+d^2(x_{\bm{\tau}}^k,z_{\bm{\tau}}^k)\}
 +2(1-K')|\bm{\tau}|\Delta_{\ve}^T(z_0) \\
&\quad +t_{\bm{\tau}}^N \cdot O(\sqrt{|\bm{\tau}|}).
\end{align*}
Notice that
\[ \sum_{k=1}^N \tau_k \{\phi(x_{\bm{\tau}}^{k-1})-\phi(x_{\bm{\tau}}^k)\}
 \le \sum_{k=1}^N |\bm{\tau}| \{\phi(x_{\bm{\tau}}^{k-1})-\phi(x_{\bm{\tau}}^k)\}
 \le \left\{\phi(z_0)-\inf_X \phi \right\}|\bm{\tau}|. \]
Moreover, since $\phi(z_{\bm{\tau}}^{k-1})-\phi(z_{\bm{\tau}}^k) \ge -2\delta_{\bm{\tau}}^k(z_0)$,
\begin{align*}
\sum_{k=1}^N \tau_k \{\phi(z_{\bm{\tau}}^{k-1})-\phi(z_{\bm{\tau}}^k)\}
&\le \sum_{k=1}^N |\bm{\tau}| \{\phi(z_{\bm{\tau}}^{k-1})-\phi(z_{\bm{\tau}}^k)
 +2\delta_{\bm{\tau}}^k(z_0)\} \\
&\le \left\{\phi(z_0)-\inf_X \phi+2\Delta_{\ve}^T(z_0) \right\}|\bm{\tau}|.
\end{align*}
Therefore, letting $i \to \infty$ in the sequence $\mathscr{P}_{\bm{\tau}_i}$
for which $\bm{\bar{z}}_{\bm{\tau}_i}$ converges to $\zeta$, we find
\[ d^2\big( \xi(T),\zeta(T) \big) \le -2\lambda \int_0^T d^2(\xi,\zeta) \,dt. \]
This implies that the nonnegative function $f(T):=\int_0^T d^2(\xi,\zeta) \,dt$ satisfies
$f' \le -2\lambda f$ and hence $(e^{2\lambda T} f(T))' \le 0$.
Thus $f \equiv 0$ and we complete the proof of $\zeta=\xi$ and Theorem~\ref{th:TK}
(recall the paragraph following Lemma~\ref{lm:bound}).
$\qedd$

{\small

}

\end{document}